\documentclass{amsproc}

\usepackage{amsmath, amsthm, amsfonts, amssymb, nicefrac}
\usepackage{xy}
\input{xypic}
\usepackage{url, hyperref}
\usepackage{colonequals}
\usepackage{color}
\definecolor{mylinkcolor}{rgb}{0.55,0.0,0.0}
\definecolor{myurlcolor}{rgb}{0.0,0.0,0.55}
\hypersetup{colorlinks=true,urlcolor=myurlcolor,citecolor=myurlcolor,linkcolor=mylinkcolor,linktoc=page,breaklinks=true}

\newtheorem{theorem}{Theorem}[section]
\newtheorem{lemma}[theorem]{Lemma}

\theoremstyle{definition}
\newtheorem{definition}[theorem]{Definition}
\newtheorem{example}[theorem]{Example}
\newtheorem{proposition}[theorem]{Proposition}

\theoremstyle{remark}
\newtheorem{remark}[theorem]{Remark}

\numberwithin{equation}{section}

\newcommand{\Aut}{\operatorname{Aut}}

\newcommand{\End}{\operatorname{End}}
\newcommand{\et}{\mathrm{et}}
\newcommand{\Exp}{\mathrm{E}}

\newcommand{\Gal}{\mathrm{Gal}}

\newcommand{\Jac}{\operatorname{Jac}}
\newcommand{\lcm}{\operatorname{lcm}}

\newcommand{\topo}{\operatorname{top}}
\newcommand{\Trace}{\operatorname{Trace}}

\newcommand{\GL}{\mathrm{GL}}
\newcommand{\GSp}{\mathrm{GSp}}
\newcommand{\Hg}{\mathrm{Hg}}
\newcommand{\Lef}{\operatorname{L}}
\newcommand{\MT}{\mathrm{MT}}

\newcommand{\PSL}{\mathrm{PSL}}
\newcommand{\PSU}{\mathrm{PSU}}

\newcommand{\SO}{\mathrm{SO}}
\newcommand{\Sp}{\mathrm{Sp}}
\newcommand{\SU}{\mathrm{SU}}
\newcommand{\ST}{\mathrm{ST}}
\newcommand{\RM}{\mathrm{RM}}
\newcommand{\QM}{\mathrm{QM}}
\newcommand{\CM}{\mathrm{CM}}
\newcommand{\IM}{\mathrm{IM}}
\newcommand{\Sym}{\operatorname{Sym}}
\newcommand{\TL}{\operatorname{TL}}
\newcommand{\Unitary}{\mathrm{U}}
\newcommand{\USp}{\mathrm{USp}}

\newcommand{\cyc}[1]{{\mathrm{C}_{#1}}}
\newcommand{\cycsup}[2]{{\mathrm{C}_{#1}^{#2}}}
\newcommand{\dih}[1]{{\mathrm{D}_{#1}}}
\newcommand{\alt}[1]{{\mathrm{A}_{#1}}}
\newcommand{\sym}[1]{{\mathrm{S}_{#1}}}

\newcommand{\bA}{{\mathbf{A}}}
\newcommand{\bB}{{\mathbf{B}}}
\newcommand{\bC}{{\mathbf{C}}}
\newcommand{\bD}{{\mathbf{D}}}
\newcommand{\bE}{{\mathbf{E}}}
\newcommand{\bF}{{\mathbf{F}}}
\newcommand{\bG}{{\mathbf{G}}}
\newcommand{\bH}{{\mathbf{H}}}
\newcommand{\bI}{{\mathbf{I}}}
\newcommand{\bJ}{{\mathbf{J}}}
\newcommand{\bK}{{\mathbf{K}}}
\newcommand{\bL}{{\mathbf{L}}}
\newcommand{\bM}{{\mathbf{M}}}
\newcommand{\bN}{{\mathbf{N}}}

\newcommand{\C}{\mathbb C}
\newcommand{\F}{\mathbb F}

\newcommand{\M}{\mathrm{M }}

\newcommand{\Q}{\mathbb Q}
\newcommand{\Qbar}{{\overline{\mathbb Q}}}
\newcommand{\R}{\mathbb R}
\newcommand{\Z}{\mathbb Z}
\newcommand{\Gap}{\textsc{Gap}}
\newcommand{\Magma}{\textsc{Magma}}

\newcommand\acc[2]{\ensuremath{{}^{#1}\hskip-0.3ex{#2}}}

\newcommand{\ttimes}{{\hspace{-0.5pt}\times\hspace{0pt}}}

\begin{document}

\title[Sato--Tate groups of abelian threefolds: a preview]{Sato--Tate groups of abelian threefolds: \\
a preview of the classification}

\author{Francesc Fit\'e}
\address{Department of Mathematics,
Massachusetts Institute of Technology,
77 Massachusetts Ave., Cambridge, MA 02139, USA}
\email{ffite@mit.edu}
\thanks{Fit\'e was supported by the Institute for Advanced Study during 2018--2019 via NSF grant DMS-1638352, and by the Simons Collaboration in Arithmetic Geometry, Number Theory, and Computation (Simons Foundation grant 550033).}

\author{Kiran S. Kedlaya}
\address{Department of Mathematics \\ University of California, San Diego \\ 9500 Gilman Drive \#0112 \\ 
La Jolla \\ CA 92093 \\ USA}
\email{kedlaya@ucsd.edu}
\thanks{Kedlaya was financially supported by the NSF (grant DMS-1802161), UC San Diego (Stefan E. Warschawski professorship), and IAS (visiting professorship 2018--2019).}

\author{Andrew V. Sutherland}
\address{Department of Mathematics,
Massachusetts Institute of Technology,
77 Massachusetts Ave., Cambridge, MA 02139, USA}
\email{drew@math.mit.edu}
\thanks{Sutherland was financially supported by the NSF (grant DMS-1522526) and by the Simons Collaboration in Arithmetic Geometry, Number Theory, and Computation (Simons Foundation grant 550033).}

%    General info
\date{\today}

%\dedicatory{This paper is dedicated to ...}


\begin{abstract}
We announce the classification of Sato--Tate groups of abelian threefolds over number fields; there are 410 possible conjugacy classes of closed subgroups of $\USp(6)$ that occur. 
We summarize the key points of the ``upper bound'' aspect of the classification, and give
a more rigorous treatment of the ``lower bound'' by realizing 33 groups that appear in the classification as maximal cases with respect to inclusions of finite index.
Further details will be provided in a subsequent paper.
\end{abstract}

\maketitle

%\tableofcontents

\section{Introduction}

The \emph{Sato--Tate group} of an abelian variety $A$ over a number field $k$ is a certain compact Lie group which conjecturally governs the distribution of normalized Euler factors of the $L$-function of $A$. For example,
if $A$ is an elliptic curve, then the Sato--Tate group is determined by whether or not $A$ has complex multiplication (and if so, whether the field of complex multiplication is contained in $k$).
In the generic case where $A$ has no complex multiplication, its Sato--Tate group is $\SU(2)$ and the usual Sato--Tate conjecture predicts that the distribution of normalized Frobenius traces of $A$ in the interval $[-2,2]$ is the same as the distributions of traces of random matrices in $\SU(2)$ determined by the Haar measure on $\SU(2)$; this is known when~$k$ is a totally real field \cite{BLGG11}
or a CM field \cite{AAC18}.

The present paper constitutes an announcement of the classification of Sato--Tate groups of abelian threefolds over number fields, in the style of the classification of Sato--Tate groups of abelian surfaces made in \cite{FKRS12}
(building upon \cite{KS08, KS09}).
These groups are closed subgroups of $\USp(6)$ that are defined up to conjugacy; 
we find that there are 410 Sato--Tate groups that occur for abelian threefolds
(compared to 3 for elliptic curves and 52 for abelian surfaces).

Given the dramatically larger size of the classification in dimension 3, it is not feasible to give a complete account here;
we have thus structured this article as a preview of our subsequent paper \cite{FKS} in which we give a complete treatment. Here, we focus on the following steps.

\begin{itemize}
\item
In \S\ref{section: Sato--Tate groups}, we recall some theoretical background needed to make a rigorous definition of the Sato--Tate group; recall a twisting construction that will be needed to realize certain candidate Sato--Tate groups;
and formulate some necessary conditions on Sato--Tate groups that reduce the classification to a finite problem (the \emph{Sato--Tate axioms}).
\item
In \S\ref{section: upper bound}, we state the classification of closed subgroups of $\USp(6)$ satisfying the Sato--Tate axioms to be obtained in \cite{FKS} (Theorem~\ref{theorem: upper bound}); there are 433 such groups.
We identify the 14 connected groups that can occur, corresponding to the possible geometric endomorphism $\R$-algebras of an abelian threefold over an algebraically closed field of characteristic $0$. For each such group $G^0$, we describe the maximal finite-index overgroups of $G^0$ that occur; the derivation of this description is given completely in all cases except
when $G^0$ is a one-dimensional torus. In this case, we state a uniform presentation of the maximal overgroups
using complex reflection groups, and summarize how the description of finite subgroups of $\PSU(3)$ by Blichfeldt--Dickson--Miller \cite{MBD61} is used to confirm this list in \cite{FKS}.

\item
In \S\ref{section: realization}, we give a complete proof (modulo enumeration of groups; see below) that within the list of 433 groups from
Theorem~\ref{theorem: upper bound}, exactly 410 occur as Sato--Tate groups (Theorem~\ref{theorem: lower bound}).
More precisely, we show that exactly 23 groups must be omitted; these groups have identity components
containing a factor which is a two-dimensional or three-dimensional torus, and the exclusion of these groups
requires a careful study of Shimura's theory of CM types analogous to \cite[\S 4.3]{FKRS12}.
Revising the list of maximal subgroups enumerated in \S\ref{section: upper bound} based on this exclusion
yields a list of 33 candidate groups which are maximal with respect to inclusions of finite index.
Using products and twisting constructions, we show that each maximal group
arises from some abelian threefold over $\Q$ or $\Q(\sqrt 3)$; by base extension, we may obtain examples of abelian threefolds (over various number fields) realizing all candidate groups.

\end{itemize}

We defer to \cite{FKS} the discussion of the following aspects, which for abelian surfaces are treated in \cite{FKRS12}.
\begin{itemize}
\item
Giving a detailed proof of Theorem~\ref{theorem: upper bound}. This amounts to confirming that every group that can occur is contained in one of the announced maximal groups.

\item
Enumerating the finite-index subgroups of the announced maximal groups in order to verify that they represent 433 distinct conjugacy classes within $\USp(6)$. This is needed to confirm that the lower bound in Theorem~\ref{theorem: upper bound} is indeed equal to 410.

\item
Computation of moment statistics and point densities. For abelian surfaces, all 52 Sato--Tate groups can be distinguished by their moments, but this fails for threefolds.
\item
Rigorous determination of the Sato--Tate group of an explicitly specified abelian threefold. One approach to this is to compute
endomorphism algebras using the method of \cite{CMSV19}.
\item
Large-scale surveys of Sato--Tate groups. With the computational resources and software tools now available \cite{Sut18a,Sut18b}, this can now be done at much greater scale than in \cite{FKRS12}.
\item
Distinguishing between Jacobians of curves, principally polarized abelian threefolds, and arbitrary abelian threefolds.
For abelian surfaces this does not change the classification of Sato--Tate groups, as was shown in \cite{FKRS12} using a large scale survey of genus 2 curves over $\Q$.
For threefolds the situation is not yet clear.
\item
Analysis of fields of definition. For abelian surfaces, only 34 of the 52 possible groups occur over $\Q$;
the situation for abelian threefolds is unclear.
\end{itemize}

\section{Background on Sato--Tate groups}
\label{section: Sato--Tate groups}

In this section, we recall some of the theoretical results from \cite{FKRS12}
that form the basis of the classification of Sato--Tate groups of abelian surfaces, and extend these as required to handle the case of abelian threefolds. 
We then use these results to classify the possible identity components of Sato--Tate groups of abelian threefolds.
See \cite{Sut16} for an overview of this circle of ideas.

\subsection{The Mumford--Tate group}

We begin by recalling the definition of the Mumford--Tate group of a (polarized) abelian variety $A$ over a number field $k$.
This construction carries enough information to determine the identity component of the Sato--Tate group.

\begin{definition}
Fix an embedding of $k$ into $\C$
and set $V \colonequals H^1(A_{\C}^{\topo}, \Q)$; it carries an alternating pairing $\psi$ given by the cup product.
Via the description of $A_{\C}^{\topo}$ as a complex torus, we obtain an $\R$-linear identification of $V \otimes_{\Q} \R$ with the tangent space of $A_{\C}$ at the origin (or equivalently, the dual of the space of holomorphic differentials on $A_{\C}$). The \emph{Mumford--Tate group} of $A$, denoted $\MT(A)$, is the smallest $\Q$-algebraic subgroup of $\GL(V)$ whose base extension to $\R$ contains the scalar action of $\C^\times$ on
$V \otimes_{\Q} \R$; this group is contained in the symplectic group $\GSp(V, \psi)$.
It follows from Deligne's theorem on absolute Hodge cycles (see \cite{DM82})
that the definition of $\MT(A)$ does not depend on the choice of the embedding of $k$ into $\C$; it is also clearly invariant under enlargement of $k$.
\end{definition}

\begin{definition}
The \emph{Hodge group} $\Hg(A)$ is the intersection $\MT(A) \cap \Sp(V, \psi)$.
The \emph{Lefschetz group} $\Lef(A)$ is the
connected component of the identity in the centralizer of $\End(A_{\overline{k}})_{\Q}$ in $\Sp(V, \psi)$.
There is an obvious inclusion $\Hg(A) \subseteq \Lef(A)$.
\end{definition}

\begin{proposition} \label{Hodge equals Lefschetz}
For $g \leq 3$, we have $\Hg(A) = \Lef(A)$. (That is, the Mumford--Tate group is determined by 
endomorphisms.)
\end{proposition}
\begin{proof}
See \cite[Theorem~2.16]{FKRS12}.
\end{proof}

\begin{remark}
As shown by Mumford \cite{Mum69}, Proposition~\ref{Hodge equals Lefschetz} fails for $g=4$.  The case $g=3$ considered here is the last case where the endomorphism-based methods developed in \cite{FKRS12} and extended here and in \cite{FKS} will suffice.
\end{remark}

\subsection{Definition of the Sato--Tate group}

We continue with the definition of the algebraic Sato--Tate group and Sato--Tate group of an abelian variety $A$ over 
a number field $k$ in
terms of $\ell$-adic monodromy, and the statement of the refined Sato--Tate conjecture.
This material is taken primarily from \cite[\S 2.1]{FKRS12}, but includes some further results from \cite[\S 2]{BK15} and \cite{BaK16} that we note here.

We denote by $\End(A)$ the ring of endomorphisms of $A$ (defined over $k$). By the endomorphism field of $A$, we refer to the minimal extension $K$ of $k$ such that $\End(A_K)$ coincides with the geometric endomorphism ring $\End(A_{\overline k})$ of $A$. Note that $K/k$ is a finite and Galois extension.

\begin{definition}
For each $\tau \in \Gal(K/k)$, define
\[
\Lef(A, \tau) \colonequals \{\gamma \in \Sp(V, \psi): \gamma^{-1} \alpha \gamma = \tau(\alpha) \mbox{ for all } \alpha \in \End(A_{K})\}.
\]
The \emph{twisted Lefschetz group} $\TL(A)$ is the union of $\Lef(A, \tau)$ over all $\tau$.
\end{definition}

\begin{proposition} \label{MT conjecture in dim 3}
For $g \leq 3$, for any prime $\ell$, $\TL(A) \otimes_{\Q} \Q_\ell$ is the kernel of the symplectic character 
on the Zariski closure of the image of $\rho_\ell: G_k \to H^1_{\et}(A_{\overline{k}}, \Q_\ell)$.
\end{proposition}
\begin{proof}
This amounts to the statement that the Mumford--Tate conjecture holds for $A$ whenever $g \leq 3$. See for example 
\cite[Theorem~6.11]{BK15}.
\end{proof}

\begin{definition}
In light of Proposition~\ref{Hodge equals Lefschetz} and Proposition~\ref{MT conjecture in dim 3},
for $g \leq~3$, we define the \emph{Sato--Tate group} $\ST(A)$ to be a maximal compact subgroup of $\TL(A) \otimes_{\Q}~\C$.
\end{definition}

\begin{lemma}\label{lemma: Galois action of component group}
For $g \leq 3$, there is a canonical isomorphism $\ST(A)/\ST(A)^0 \to \Gal(K/k)$. In particular, this isomorphism is compatible with base change: for any finite extension $k'$ of $k$, 
$\ST(A_{k'})$ is the inverse image of $\Gal(k'K/k') \subseteq \Gal(K/k)$ in $\ST(A)$.
\end{lemma}
\begin{proof}
This is again a consequence of \cite[Theorem~6.11]{BK15}.
\end{proof}

\begin{remark} \label{remark: general ST group}
In the definition of the Sato--Tate group, we are implicitly using the fact that for $g \leq 3$, 
the Mumford--Tate group of $A$ is determined by endomorphisms. For general $g$, it is expected that the role of
$\TL(A)$ in Proposition~\ref{MT conjecture in dim 3} can be filled by a certain algebraic group over $\Q$,
the \emph{motivic Sato--Tate group}, whose definition involves algebraic cycles on $A \times A$ of all codimensions, not just endomorphisms \cite{BaK16}. There is still a canonical surjection $\ST(A)/\ST(A)^0 \to \Gal(K/k)$,
but it is not in general an isomorphism.
\end{remark}

\subsection{Twisting and the Sato--Tate group}

\begin{definition}
For $L/k$ a finite Galois extension and $\xi$ a $1$-cocycle of $\Gal(L/k)$ valued in $\Aut(A_L)$, 
there exists a unique (up to unique isomorphism) abelian variety $A_\xi$ over $k$ and an isomorphism
$\theta: A_{\xi,L} \to A_L$ such that
\[
\xi(\sigma) = \theta^\sigma \circ \theta^{-1} \qquad (\text{for all }\sigma \in \Gal(L/k)).
\]
Moreover, $A_\xi$ depends only on the class of $\xi$ in the pointed set $H^1(\Gal(L/k), \Aut(A_L))$.
\end{definition}

\begin{example}\label{example: twistcube}
Let $E$ be an elliptic curve without complex multiplication defined over $\Q$ and let $L/\Q$ be a Galois extension. Let $A=E^d$, and fix an isomorphism $\Aut(A_L)\simeq \GL_d(\Z)$ for some integer $d\geq 1$. For every rational degree~$d$ Artin representation $\rho$, there is a choice of a basis for which the image of $\rho$ lies in $\GL_d(\Z)$. We may thus regard a rational degree $d$ representation of $\Gal(L/\Q)$ as a 1-cocycle 
$$
\xi\colon \Gal(L/\Q) \rightarrow \Aut(A_L)\,.
$$
The endomorphism field of the abelian $d$-fold $A_\xi$ defined above is the field cut out by the representation $\rho\otimes \rho^{\vee}$. The Galois group of the endomorphism field is thus isomorphic to the projective image of the original degree~$d$ rational representation $\rho$ of $\Gal(L/\Q)$. We will use this construction in \S\ref{section: diagonalprodreal}.
\end{example}

\begin{example} \label{example: twistcube CM}
Similarly, let $E$ be an elliptic curve over $\Q$ with complex multiplication by the maximal order of the imaginary quadratic field $M$ (which must have class number one, since $E$ is defined over $\Q$). Let $L/\Q$ be a Galois extension containing $M$ for which the extension
\[
1 \to \Gal(L/M) \to \Gal(L/\Q) \to \Gal(M/\Q) \to 1
\]
is split.
Let $A=E^d$, and fix an isomorphism $\Aut(A_L)\simeq \GL_d(\mathcal O_M)$ for some integer $d\geq 1$. 
For the action of $\cyc 2$ on $\GL_d(M)$ via complex conjugation on $M$,
suppose that $\tilde{\rho}: \Gal(L/\Q) \to \GL_d(M) \rtimes \cyc 2$ is a representation whose restriction $\rho$ to $\Gal(L/M)$ factors through
$\GL_d(M)$.
Since $\mathcal O_M$ has class number 1, there is a choice of basis of $M^d$ for which $\rho$ factors through a representation $\rho_0: \Gal(L/M) \to \GL_d(\mathcal O_M)$.

Suppose further that the basis of $M^d$ can be chosen so that $\tilde{\rho}$ factors through
$\GL_d(\mathcal O_M) \rtimes \cyc 2$ (see Remark~\ref{remark: twisting and principal polarizations}). We then obtain from $\tilde{\rho}$ a cocycle
$$
\xi\colon \Gal(L/\Q) \rightarrow \Aut(A_L)
$$
and again the Galois group of the endomorphism field of $A_\xi$ is isomorphic to the projective image of $\Gal(L/\Q)$.
We will use this construction in \S\ref{section: diagonalprodimag}.
\end{example}

\begin{remark} \label{remark: twisting and principal polarizations}
In Example~\ref{example: twistcube}, the condition that $\tilde{\rho}$ can be factored through
$\GL_d(\mathcal O_M) \rtimes \cyc 2$ is equivalent to requiring the existence of a choice of $\rho_0$ for which
\begin{equation} \label{eq: integral compatibility}
\rho_0 \cong c_M \circ \rho_0 \circ c,
\end{equation}
where $c$ denotes the action of $\Gal(M/\Q)$ on $\Gal(L/M)$ while $c_M$ denotes complex conjugation on $\mathcal O_M^d$.
The following considerations will be useful.
\begin{itemize}
\item[(i)]
We automatically have \eqref{eq: integral compatibility} if $\rho_0$ is uniquely determined by $\rho$.
(This can only occur if $\rho$ is absolutely irreducible.)
\item[(ii)]
We also have \eqref{eq: integral compatibility} if $\tilde{\rho}$ descends to $\GL_d(\Z) \times \cyc 2$ and $\Gal(M/\Q)$ acts on $\Gal(L/M)$ as an inner automorphism.
\item[(iii)]
When $\rho_0$ is not uniquely determined by $\rho$, the choices for $\rho_0$ will be distinguished by their reductions modulo some $\alpha \in \mathcal O_M$. Using this, we can sometimes track the effects of $c$ and $c_M$ and thus detect whether or not \eqref{eq: integral compatibility} holds.
\end{itemize}
\end{remark}

\begin{remark} \label{remark: twisting with denominator}
We will encounter a small number of cases where \eqref{eq: integral compatibility} fails. In these cases,
we have an isomorphism $\rho \cong c_M \circ \rho \circ c$ which does not descend to $\rho_0$; 
that is, the image of $\rho_0$ is normalized by some matrix of $\GL_d(M) \setminus \GL_d(\mathcal O_M)$.

In general, such cases correspond to twists not of $E^d$ but of an isogenous abelian variety
(which by \cite[Theorem~2]{Kan11} is itself a product of elliptic curves isogenous to $E$).
We work out two specific constructions of this type.
(Compare \cite[\S 4]{SM74} or \cite[Example~5.15]{Ma11}.)

\begin{itemize}
\item
Let $E$ be the elliptic curve $y^2 = x^3 - 1$; $E$ has CM by the maximal order $\Z[\zeta_3]$
and contains the rational 2-torsion point $Q = (1,0)$,
with the other two 2-torsion points $Q_1, Q_2$ being defined over $M$.
Let $A$ be the quotient of the Weil restriction of $E_M$ from $M$ to $\Q$ by the rational subgroup of order 2 corresponding to $\langle Q_1 \rangle$; then
\begin{align*}
\End(A_M) &\cong
\M_2(\mathcal O_M) \cap \begin{pmatrix} 2^{-1} & 0 \\ 0 & 1 \end{pmatrix} 
\M_2(\mathcal O_M) \begin{pmatrix} 2 & 0 \\ 0 & 1 \end{pmatrix} \\
&= \left\{ \begin{pmatrix} a&b \\ c&d \end{pmatrix} \in M_2(\mathcal O_M): b \equiv 0 \pmod{2} \right\}
\end{align*}
with the action of $\Gal(M/\Q)$ given by
\[
g \mapsto \begin{pmatrix} 2^{-1} & 0 \\ 0 & 1 \end{pmatrix} \overline{g} \begin{pmatrix} 2 & 0 \\ 0 & 1 \end{pmatrix}.
\]
(Note that $A_M \cong E_M \times E'_M$ where $E'_M = E_M / \langle Q_1 \rangle$ has CM by the nonmaximal order $\Z[2\zeta_3]$.)

\item
Put $k = \Q(\sqrt{3})$. Let $E$ be the elliptic curve over $\Q(\sqrt{3})$ with LMFDB label \href{https://www.lmfdb.org/EllipticCurve/2.2.12.1/9.1/a/3}{\texttt{2.2.12.1-9.1-a3}};
it has CM by $\Z[i]$ and has a rational point $Q$ of order 3.
Let $A$ be the product $E \times E/\langle Q \rangle$; then
\begin{align*}
\End(A) &\cong \M_2(\mathcal O_M) \cap \begin{pmatrix} 3^{-1} & 0 \\ 0 & 1 \end{pmatrix} 
\M_2(\mathcal O_M) \begin{pmatrix} 3 & 0 \\ 0 & 1 \end{pmatrix} \\
&= \left\{ \begin{pmatrix} a&b \\ c&d \end{pmatrix} \in M_2(\mathcal O_M): b \equiv 0 \pmod{3} \right\}
\end{align*}
with the action of $\Gal(M/\Q)$ given by
\[
g \mapsto \begin{pmatrix} 3^{-1} & 0 \\ 0 & 1 \end{pmatrix} \overline{g} \begin{pmatrix} 3 & 0 \\ 0 & 1 \end{pmatrix}.
\]
(Note that $E/\langle Q \rangle$ has CM by the nonmaximal order $\mathbb{Z}[3i]$.)
\end{itemize}
\end{remark}

\subsection{Axioms for Sato--Tate groups}

A key tool used to classify Sato--Tate groups in \cite{FKRS12} is a list of necessary conditions
called the \emph{Sato--Tate axioms}. The formulation in \cite[Definition~3.1]{FKRS12} is applicable to
arbitrary motives; we state here a restricted form of the three original Sato--Tate axioms applicable
to the 1-motives associated to abelian threefolds. We also add a fourth axiom coming from the 
fact that the Sato--Tate group of an abelian threefold is determined by endomorphisms
(Proposition~\ref{Hodge equals Lefschetz}).

\begin{definition} \label{Sato--Tate axioms}
For a subgroup $G$ of $\GL_6(\C)$ with identity connected component $G^0$, the \emph{Sato--Tate axioms (for abelian threefolds)} are as follows.
\begin{enumerate}
\item[(ST1)]
The group $G$ is a closed subgroup of $\USp(6)$. For definiteness, we take the latter to be defined with respect to the symplectic form given by the block matrix
\[
J = \begin{pmatrix} 0 & I_3 \\
-I_3 & 0
\end{pmatrix},
\]
where $I_3$ denotes a $3\times 3$ identity matrix, unless otherwise specified.
\item[(ST2)] (Hodge condition)
There exists a homomorphism $\theta\colon \Unitary(1) \to G^0$
such that $\theta(u)$ has eigenvalues $u, u^{-1}$ of multiplicity $3$.
The image of such a $\theta$ is a \emph{Hodge circle}, and the set of all Hodge circles must generate a dense subgroup of $G^0$.
\item[(ST3)] (Rationality condition)
For each component
$H$ of $G$ and each irreducible character $\chi$ of $\GL_6(\C)$,
the expected value (under the Haar measure) of $\chi(\gamma)$ over $\gamma\in H$ is an integer.
In particular, for any positive integers $m$ and $n$, the expectation $\Exp[\Trace(\gamma, \wedge^m \C^{6})^n: \gamma \in H]$ is an integer.
\item[(ST4)] (Lefschetz condition)
The subgroup of $\USp(6)$ fixing $\End(\C^6)^{G^0}$ is $G^0$.
\end{enumerate}
\end{definition}

\begin{proposition} \label{necessity of Sato--Tate}
Let $A$ be an abelian threefold over $k$.
Then $G = \ST_A$ satisfies the Sato--Tate axioms.
\end{proposition}
\begin{proof}
For (ST1), (ST2), (ST3), this is \cite[Proposition~3.2]{FKRS12}
except that the density condition in (ST2) is not stated therein; that statement is a consequence of
the definition of the Mumford--Tate group \cite[Definition~2.11]{FKRS12}.
For (ST4), apply Proposition~\ref{Hodge equals Lefschetz}.
\end{proof}

\section{Classification: an overview}
\label{section: upper bound}

This section is mainly expository. We report on several aspects of the following theorem, whose proof will appear in the upcoming work \cite{FKS}. 

\begin{theorem}[\cite{FKS}]\label{theorem: upper bound}
Up to conjugacy, at most $433$ subgroups of $\GL_6(\C)$ satisfy the Sato--Tate axioms for an abelian threefold. 
%At most 413 of them can occur as the Sato--Tate group of an abelian threefold defined over a number field.
Among these, $30$ are maximal with respect to finite index inclusions.
\end{theorem}

\begin{remark} \label{remark: upper bound}
In \S\ref{section: realization}, we will show that all 433 groups of the above theorem satisfy the Sato--Tate axioms (see Remark \ref{remark: lowerbound}).
\end{remark}

We will describe the possible connected components of these groups (\S\ref{section: identity component}), explain their connection to the possible geometric endomorphism rings of abelian threefolds (\S\ref{section: endomorphism rings}), and for each identity component present the extensions that are maximal with respect to the relation of inclusion of finite index up to conjugacy (\S\ref{section: extensions}). We should nevertheless emphasize that \S\ref{section: identity component} and \S\ref{section: endomorphism rings} contain complete proofs; only \S\ref{section: extensions} depends on results in \cite{FKS}.

\subsection{Connected component of the identity}\label{section: identity component}

We first identify the connected groups that satisfy the Sato--Tate axioms (ST1), (ST2), (ST4). Note that condition (ST3) is vacuous for a connected group, as the expected value in question is simply the dimension of the fixed subspace of the representation with character~$\chi$.

Let $G^0$ be a connected group satisfying the Sato--Tate axioms and  
let $T$ be a maximal torus of $G^0$. Let $\mathfrak{h}$ denote the Lie algebra of $G^0$ and let $\mathfrak{h}_\C$ be the complexification of $\mathfrak{h}$. Since $T$ is contained in a maximal torus of $\USp(6)$, which is 3-dimensional, $\mathfrak{h}_\C$ has rank at most 3.

The set of Lie algebras of rank at most 3 can easily be determined from the classification of Dynkin diagrams:

$$
\begin{array}{|c|c|}\hline
r & \mbox{Lie algebras of rank }r \\\hline
1 & \mathfrak{t}_1,\ \mathfrak{sl}_2 = \mathfrak{so}_3 \\\hline
2 & \mathfrak{t}_2,\ \mathfrak{t}_1 \ttimes \mathfrak{sl}_2,\ \mathfrak{sl}_2 \ttimes \mathfrak{sl}_2 = \mathfrak{so}_4,\ \mathfrak{sl}_3,\  \mathfrak{sp}_4 = \mathfrak{so}_5,\ \mathfrak{g}_2 \\\hline	
3 & 
\begin{array}{c}
\mathfrak{t}_3, \mathfrak{t}_2 \ttimes \mathfrak{sl}_2,\ 
\mathfrak{t}_1 \ttimes \mathfrak{sl}_2 \ttimes \mathfrak{sl}_2,\ 
\mathfrak{t}_1 \ttimes \mathfrak{sl}_3, \mathfrak{t}_1 \ttimes \mathfrak{sp}_4, \mathfrak{t}_1 \ttimes \mathfrak{g}_2,\ \mathfrak{sl}_2 \ttimes \mathfrak{sl}_2 \ttimes \mathfrak{sl}_2, \\
\mathfrak{sl}_2 \ttimes \mathfrak{sl}_3,\ \mathfrak{sl}_2 \ttimes \mathfrak{sp}_4, \mathfrak{sl}_2 \ttimes \mathfrak{g}_2,\ \mathfrak{sl}_4 = \mathfrak{so}_6,\  \mathfrak{so}_7,\ \mathfrak{sp}_6
\end{array}\\\hline
\end{array}
$$
\smallskip

The standard representation of $\USp(6)$ gives rise to a faithful 6-dimensional unitary symplectic self-dual representation
of $\mathfrak{h}$. This eliminates some Lie algebras from the above list.
\begin{itemize}
\item
There are no 6-dimensional symplectic representations of $\mathfrak{sp}_4$ except if we allow some trivial
summands, but these violate (ST2).
\item
The smallest dimension of a nontrivial
representation of $\mathfrak{g}_2$ is $7 > 6$; this also rules out $\mathfrak{t}_1 \times \mathfrak{g}_2$
and $\mathfrak{sl}_2 \times \mathfrak{g}_2$.
\item
The only nontrivial self-dual representation of $\mathfrak{sl}_3$ of dimension at most~$6$ is the
sum of the standard 3-dimensional representation and its dual.
This rules out $\mathfrak{sl}_2 \times \mathfrak{sl}_3$.
\item
The only nontrivial self-dual representation of $\mathfrak{so}_6$ is the standard representation,
which is not symplectic.
\item
The smallest dimension of a nontrivial representation of $\mathfrak{so}_7$ is $7>6$.
\end{itemize}

The above considerations, together with an additional argument to rule out the group $\Unitary(2)$ (see \cite[Rem. 2.3]{FKS16}), allow us to conclude that $G^0$ must correspond to one of the following groups (the subscripts in the notation are defined below):

$$
\begin{array}{|c|c|}\hline
\dim(T) & \mbox{Connected groups satisfying (ST1) and (ST2)} \\\hline
1 & \Unitary(1)_3,\ \SU(2)_3,\ \SO(3) \\\hline
2 & \Unitary(1) \ttimes \Unitary(1)_2,\ \Unitary(1) \ttimes \SU(2)_2,\ \SU(2) \ttimes \Unitary(1)_2,\ \SU(2) \ttimes \SU(2)_2,\ \SU(3) \\\hline
3 & 
\begin{array}{c}
\Unitary(1) \ttimes \Unitary(1) \ttimes \Unitary(1),\ \Unitary(1) \ttimes \Unitary(1) \ttimes \SU(2),\ \Unitary(1) \ttimes \SU(2) \ttimes \SU(2),\ \Unitary(3), \\
\Unitary(1) \ttimes \USp(4),\ \SU(2) \ttimes \SU(2) \ttimes \SU(2),\ \SU(2) \ttimes \USp(4),\ 
\USp(6)
\end{array}\\\hline
\end{array}
$$
\smallskip

Each group comes equipped with a 6-dimensional representation,
which may be inferred from the preceding table using the following considerations:
\begin{itemize}
\item The groups $\SU(2)$, $\USp(4)$, $\USp(6)$ carry their standard representations. 
\item For $d=2,3$, the notations $\Unitary(1)_d, \SU(2)_d$ refer to the $d$-fold diagonal representations of $\Unitary(1), \SU(2)$.
\item The groups $\Unitary(1)$, $\Unitary(3)$ are embedded into $\SU(2), \USp(6)$ by the formula
\begin{equation}\label{equation: embed}
A \mapsto \begin{pmatrix} A & 0 \\ 0 & \overline{A} 
\end{pmatrix}.
\end{equation}
\item The groups $\SO(3)$ and $\SU(3)$ are viewed in $\USp(6)$ via the above embedding of $\Unitary(3)$.
\item Note that the abstract group $\Unitary(1) \times \SU(2)$ is considered twice: once as a subgroup of
$\Unitary(1) \times \Unitary(1) \times \SU(2)$ via the diagonal embedding of $\Unitary(1)$,
and once as a subgroup of $\Unitary(1) \times \SU(2) \times \SU(2)$ via the diagonal embedding
of $\SU(2)$. 
\end{itemize}

In the case $\USp(6)$ and all of the product cases, one easily verifies that (ST4) is satisfied by following the methods of
\cite[\S 4]{FKRS12}. This leaves the cases of $\SO(3)$, $\SU(3)$, and $\Unitary(3)$, which we work out by hand.
Recall that $\Unitary(3)$ is embedded into $\USp(6)$ via \eqref{equation: embed}.
As a result, we see that the endomorphisms of $\C^6$ commuting with $\SO(3)$ or $\SU(3)$ are those given by block scalar matrices, all of which also commute with $\Unitary(3)$. It follows that $\SO(3)$ and $\SU(3)$ do not satisfy (ST4). 

The possible values for a connected $G$ satisfying the Sato--Tate axioms are listed in the fourth column of Table~\ref{table: connected ST groups}.

\subsection{Geometric endomorphism algebras}
\label{section: endomorphism rings}

In this section we recover the possible geometric $\R$-algebras of endomorphisms of abelian threefolds from our classification of identity components of Sato--Tate groups. We then describe the possible geometric $\Q$-algebras of endomorphisms they correspond, by using Shimura's refinement of Albert's classification of division algebras equipped with a positive involution. 

\subsubsection{Geometric endomorphism $\R$-algebras} For an abelian threefold $A$ defined over a number field $k$, we can recover the geometric endomorphism $\R$-algebra $\End(A_{\overline k})_\R\colonequals\End(A_{\overline k})\otimes \R$ by first finding the subalgebra of $\End(\C^6)$ commuting with $\ST(A)^0$, then picking out the maximal positive definite subspace of this algebra under the Rosati form \cite[Prop.\ 2.19]{FKRS12}.

In the case of $\USp(6)$ and all the other product cases, the above computation is easily carried out with the methods of \cite[\S4]{FKRS12}.
In the case of $\Unitary(3)$ the Rosati form is given up to scalars by
\[
\begin{pmatrix}
uI_3 & 0 \\ 0 & vI_3
\end{pmatrix}
\mapsto 6uv = \frac{3}{2}((u+v)^2 - (u-v)^2),
\]
so the maximal positive definite subspace consists of those matrices for which $u+v \in \R,
u-v \in i\R$, or equivalently $v = \overline{u}$. Therefore $\End(A_{\overline k})_\R\simeq \C$. We assign letters $\bA$, $\bB$, \ldots, $\bN$ to each of the possibilities for $\End(A_{\overline k})_ \R$. We refer to this letter as the \emph{absolute type} of $A$. The absolute types together with the possible values for $\End(A_{\overline k})_\R$ are listed in the first and second columns of Table~\ref{table: connected ST groups}.

\subsubsection{Geometric endomorphism $\Q$-algebras} For a simple abelian variety $A$ of dimension $g \leq 3$, the work of Albert and Shimura leaves the following possibilities for the geometric $\Q$-algebra of endomorphisms $\End(A_{\overline k})_\Q:=\End(A_{\overline k})\otimes \Q$. Note that for a fixed $g\leq 3$, each possibility is distinguished by the $\Q$-dimension $d$ of $\End(A_{\overline k})_\Q$ (equal to the $\R$-dimension of $\End(A_{\overline k})_\R$).
(See for example \cite{Oo87}.)
\medskip

\begin{center}
\begin{tabular}{|l|l|l|l|}\hline
$g$ & $d$ & Albert type & $\Q$-algebra\\\hline
$1$ & 1 & I(1) & $\Q$\\
    & 2 & IV(1,1) & imaginary quadratic field\\\hline
$2$ & 2 &  I(1) & real quadratic field\\
    & 4 & I(2) & indefinite quaternion algebra\\
    & 6 &  IV(2,1) & quartic CM field\\\hline
$3$ & 1 & I(1) & $\Q$\\
    & 2 & IV(1,1) & imaginary quadratic field\\
    & 3 & I(3) & totally real cubic field\\
    & 6 & IV(3,1) & sextic CM field\\\hline
\end{tabular}
\end{center}
\medskip

From the list of geometric $\Q$-algebra endomorphism types for simple abelian varieties of dimension $g\le 3$ one can determine all possibilities for $g=3$ by taking all possible products (including self products) that yield an abelian variety of dimension~$3$; there are 22 ways to do this.  Each yields a distinct type of $\Q$-algebra, but only 14 distinct $\R$-algebras (for example, taking three non-isogenous elliptic curves of type I(1) yields the same $\R$-algebra $\R\times\R\times \R$ as a single RM abelian threefold of type I(3)).
This leads to the following description of absolute types in terms of the the isogeny decomposition of $A_{\overline k}$ (compare with \cite[\S4, Table 2]{Sut16}):
\begin{enumerate}
\item[($\bA$)] $A_{\overline k}$ is simple of type I(1);
\item[($\bB$)] $A_{\overline k}$ is simple of type IV(1,1);
\item[($\bC$)] $A_{\overline k}$ is isogenous to the product of an elliptic curve without CM and a simple abelian surface of type I$(1)$;
\item[($\bD$)] $A_{\overline k}$ is isogenous to the product of an elliptic curve with CM and a simple abelian surface of type I$(1)$;
\item[($\bE$)] $A_{\overline k}$ is either
\begin{itemize}
\item simple of type I(3), or
\item isogenous to the product of an elliptic curve without CM and a simple abelian surface of type I(2), or
\item isogenous to the product of three pairwise non-isogenous elliptic curves without CM;
\end{itemize}
\item[($\bF$)] $A_{\overline k}$ is either
\begin{itemize}
\item isogenous to the product of an elliptic curve with CM and two nonisogenous elliptic curves without CM, or
\item isogenous to the product of an elliptic curve with CM and a simple abelian surface of type I(2);
\end{itemize}
\item[($\bG$)] $A_{\overline k}$ is either
\begin{itemize}
\item isogenous to the product of an elliptic curve without CM and two nonisogenous elliptic curves with CM, or
\item isogenous to the product of an elliptic curve without CM and a simple abelian surface of type IV(2,1);
\end{itemize}
\item[($\bH$)] $A_{\overline k}$ is either
\begin{itemize}
\item simple of type IV(3,1), or
\item isogenous to the product of an elliptic curve with CM and a simple abelian surface of type IV(2,1), or
\item isogenous to the product of three pairwise non-isogenous elliptic curves with CM; 
\end{itemize}
\item[($\bI$)] $A_{\overline k}$ is either
\begin{itemize}
\item isogenous to the product of an elliptic curve without CM and a simple abelian surface of type II(1), or
\item isogenous to the product of an elliptic curve $E$ without CM and the square of an elliptic curve without CM nonisogenous to $E$;
\end{itemize}
\item[($\bJ$)] $A_{\overline k}$ is either
\begin{itemize}
\item isogenous to the product of an elliptic curve with CM and the square of an elliptic curve without CM, or
\item isogenous to the product of an elliptic curve with CM and a simple abelian surface of type II(1);
\end{itemize}
\item[($\bK$)] $A_{\overline k}$ is isogenous to the product of an elliptic curve without CM and the square of an elliptic curve with CM;
\item[($\bL$)] $A_{\overline k}$ is isogenous to the product of an elliptic curve $E$ with CM and the square of an elliptic curve with CM nonisogenous to $E$;
\item[($\bM$)] $A_{\overline k}$ is isogenous to the cube of an elliptic curve without CM;
\item[($\bN$)] $A_{\overline k}$ is isogenous to the cube of an elliptic curve with CM.
\end{enumerate}

\begin{remark}
The third column of Table \ref{table: connected ST groups} summarizes the discussion above.  In the table one can see that absolute types are ordered according to the lexicographic ordering of pairs $(r, d)$ where $r$ is the rank of the N\'eron-Severi group of $A_{\overline{k}}$ and $d$ is the $\R$-dimension of $\End(A_{\overline k})_\R$. See \cite{CFS19} for an interpretation of $r$ and $d$ in terms of the moments of the Sato--Tate group.
\end{remark}

\begin{table}[h]
\small
\[
\begin{array}{|c|c|c|c|}\hline
\mbox{Type} & \mathrm{End}(A_{\overline k})_{\R}& \mbox{Splitting of }A_{\overline k} & \ST(A)^0   \\
\hline
\bA & \R & T & \USp(6) \\\hline
\bB & \C & S_{\IM}& \Unitary(3) \\\hline
\bC & \R \times \R & E\times S & \SU(2) \times \USp(4)  \\\hline
\bD & \C \times \R & E_{\CM}\times S &\Unitary(1) \times \USp(4)  \\\hline
\bE & \R \times \R \times \R & \begin{array}{c} T_{\RM}\\ E\times S_{\RM} \\ E\times E' \times E'' \end{array}  &  \SU(2) \times \SU(2) \times \SU(2) \\\hline
\bF & \C \times \R \times \R & \begin{array}{c} E_{\CM}\times E'\times E''\\ E_{\CM}\times S_{\RM} \end{array} & \Unitary(1) \times \SU(2) \times \SU(2)  \\\hline
\bG & \C \times \C \times \R & \begin{array}{c}E_{\CM}\times E_{\CM}'\times E''\\  S_{\CM}\times E\end{array}& \Unitary(1) \times \Unitary(1) \times \SU(2)  \\\hline
\bH & \C \times \C \times \C & \begin{array}{c} T_{\CM}\\ E_{\CM}\times S_{\CM} \\ E_{\CM}\times E'_{\CM}\times E''_{\CM} \end{array} & \Unitary(1) \times \Unitary(1) \times \Unitary(1)  \\\hline
\bI & \R \times \M_2(\R) & \begin{array}{c} E\times (E')^2\\ E\times S_{\QM}\end{array} & \SU(2) \times \SU(2)_2  \\\hline
\bJ & \C \times \M_2(\R) & \begin{array}{c} E_{\CM}\times (E')^ 2\\ E_{\CM}\times S_{\QM}\end{array} & \Unitary(1) \times \SU(2)_2 \\\hline
\bK &  \R\times \M_2(\C) & \begin{array}{c} E\times (E'_{\CM})^ 2\\ \end{array}& \SU(2) \times \Unitary(1)_2  \\\hline
\bL & \C \times \M_2(\C) & E_{\CM}\times (E'_{\CM})^2 & \Unitary(1) \times \Unitary(1)_2  \\\hline
\bM &  \M_3(\R) & E^ 3  & \SU(2)_3  \\\hline
\bN & \M_3(\C) & E_{\CM}^3 & \Unitary(1)_3 \\
\hline
\end{array}
\]
\vspace{6pt}
\caption{Real endomorphism algebras and connected parts of Sato--Tate groups of abelian threefolds. We describe the decomposition up to isogeny of $A_{\overline k}$; when doing so, we denote by $E$, $S$, $T$ simple abelian varieties over $\overline k$ of respective dimensions 1, 2, 3. Given $E$, we denote  by $E'$ an elliptic curve not isogenous to $E$ (a similar convention applies to $E''$). For $B\in \{E,S,T\}$, we write $B_{\CM}$ to indicate that $B$ has complex multiplication by a CM field of degree $2\dim(B)$ and $B_{\RM}$ to indicate that $B$ has real multiplication by a field of degree $\dim(B)$. The lack of a subindex indicates that $B$ has trivial endomorphism ring. We write $S_{\IM}$ (resp.\ $S_{\QM}$) to indicate that $S$ has multiplication by an imaginary quadratic field (resp.\ by an indefinite quaternion algebra over $\Q$).  
}\label{table: connected ST groups}
\end{table}

\subsection{Extensions}\label{section: extensions}

Let $G^0$ be one of the connected groups obtained in \S\ref{section: identity component}. Let~$N$ and~$Z$ denote the normalizer and centralizer of $G^0$ within $\USp(6)$. The general strategy in \cite{FKS} to prove Theorem \ref{theorem: upper bound} exploits the bijection between:
\begin{itemize}
\item The set of (conjugacy classes of) closed subgroups $G$ of $\USp(6)$ with identity component $G^0$ that satisfy (ST3);
\item The set of (conjugacy classes of) finite subgroups $H$ of $N/G^0$ such that $HG^0$ satisfies (ST3).
\end{itemize}
In this section we report on some aspects of this proof of Theorem \ref{theorem: upper bound}.
We will organize the discussion into four cases depending on the structure of the normalizer~$N$. We say that $G^0$ is \emph{decomposable} if there is a subgroup $G^0_1$ of $\SU(2)$ and a subgroup $G^0_2$ of $\USp(4)$ such that
\begin{equation}\label{equation: decomposable}
G^0= G^ 0_1 \times G^ 0_2\,.
\end{equation}
The indecomposable cases correspond to $G^0=\USp(6),\,\Unitary(3)$. Suppose that $G^0$ is decomposable and let $N_1$ (resp.\ $N_2$) denote the normalizer of $G^0_1$ (resp.\ $G_2^ 0$) in $\SU(2)$ (resp.\ $\USp(4)$). We say that that $G^0$ is a \emph{split product} (resp.\ a \emph{non-split product}) if 
$$
N= N_1\times N_2\,,\qquad \text{(resp. $N \supsetneq N_1\times N_2$)}.
$$
Eight of the possibilities for $G^0$ happen to be split products.
The non-split product cases include the \emph{triple products} $\SU(2)\times \SU(2)\times\SU(2)$ and $\Unitary(1)\times \Unitary(1)\times \Unitary(1)$, and also the \emph{diagonal products} $\SU(2)_3$ and $\Unitary(1)_3$. 

\begin{remark} We will only impose (ST3) explicitly in the case of diagonal products (and, in fact, only a relaxed version of (ST3)). However, it should be noted that (ST3) will be imposed in an implicit way for the split products. We will postpone to \S\ref{section: realization} the verification that (ST3) is satisfied for all the extensions described in this section. In particular, it will follow from the discussion in \S\ref{section: realization} that, when $G^0$ is indecomposable or a triple product, any closed subgroup $G$ of $N$ with identity component $G^0$ automatically satisfies (ST3). 
\end{remark}

The main purpose of this section is to describe the 30 group extensions that potentially satisfy the Sato--Tate axioms and are maximal with respect to finite index inclusions (up to conjugacy). The distribution of these groups among the possibilities for $G^0$ will be discussed below and it is summarized in Table \ref{table: extensions}.

\subsubsection{Indecomposable cases}\label{section: indecomposable cases}

In case $G^0 = \USp(6)$ (type $\bA$), it is clear that $G = G^0$, which is obviously maximal. In case $G^0 = \Unitary(3)$
(type $\bB$), we have
$Z = \Unitary(1)_3 \subset G^0$ and $N = G^0 \cup J G^0$.
Consequently, $N/G^0 \simeq \cyc 2$, so $G$ equals either $G^0$ or $N$, which is the only maximal group for the identity component $G^0=\Unitary(3)$.

\subsubsection{Split products}\label{section: splitproductsclass}
Note that in this case 
$$
G=G_1 \times G_2\,,
$$
where $G_1$ is a subgroup of $N_1$ with identity component $G_1^0$ and and $G_2$ is a subgroup of $N_2$ with identity component $G_2^0$. Here $G^0_i$ is as in \eqref{equation: decomposable}.
We start with the cases where we must have $G_i = G^0_i$ for one $i$, in which case the options come directly from the other factor.

\begin{itemize}
\item
In the case $G^0 = \SU(2) \times \USp(4)$ (type $\bC$), we must have $G_i = G_i^0$ for $i=1,2$, so $G = G^0$. This yields a maximal group.

\item
In the case $G^0 = \Unitary(1) \times \USp(4)$ (type $\bD$), we must have $G_2 = G^0_2$, so $G = G_1 \times \USp(4)$
for $G_1 \in \{ \Unitary(1), N(\Unitary(1))\}$. Then $G_1 = N(\Unitary(1))$ yields a maximal group.

\item
In the case $G^0 = \SU(2) \times \SU(2)_2$ (type $\bI$), we must have $G_1 = G_1^0$, so $G = \SU(2) \times G_2$ for some
$G_2$ among the groups of dimension $d=3$ in \cite[Table~8]{FKRS12}. Among these groups, the maximal ones are $J(E_4)$ and $J(E_6)$. For this identity component we find 10 extensions, among which two are maximal, namely $\SU(2)\times J(E_4)$ and $\SU(2)\times J(E_6)$.

\item
In the case $G^0 = \SU(2) \times \Unitary(1)_2$ (type $\bK$), we must have $G_1 = G^0_1$, so $G = \SU(2) \times G_2$ for some $G_2$
among the groups of dimension $d=1$ in \cite[Table~8]{FKRS12}. 
Among these groups, the maximal ones are $J(D_6)$ and $J(O)$. For this identity component we find 32 extensions, among which two are maximal, namely $\SU(2)\times J(D_6)$ and $\SU(2)\times J(O)$.

\item
In the case $G^0 =  \Unitary(1) \times \Unitary(1)\times \SU(2)$ (type $\bG$), take $G^0_1=\SU(2)$ and $G^0_2=\Unitary(1)\times\Unitary(1)$; we must have $G_1 = G^0_1$,
so $G = \SU(2)\times G_2$ for some $G_2$ among the groups with dimension $d=2$ in \cite[Table~8]{FKRS12}. Among these groups, only $F_{a,b,c}$ is maximal. For this identity component we find~$8$ extensions, of which one is maximal: $\SU(2)\times F_{a,b,c}$ (we will return to this case in \S\ref{section: upperbound}).
\end{itemize}

\begin{remark}
As indicated by the asterisks that appear in \cite[Table~8]{FKRS12}, the subgroup $F_{a,b,c}$ of $\USp(4)$ is one of three subgroups that satisfy the Sato--Tate axioms for abelian surfaces but do not arise as the Sato--Tate group of an abelian surface.  This explains the discrepancy between the counts 55 and 52 in \cite{FKRS12} of groups that satisfy the Sato--Tate axioms and groups that arise as Sato--Tate groups, and it increases the number of maximal Sato--Tate groups with identity component $\Unitary(1)\times \Unitary(1)$ (two of the maximal subgroups of $F_{a,b,c}$ arise as Sato--Tate groups).  A similar phenomenon occurs here, as explained in Remark~\ref{remark: maximalgroups} below.
\end{remark}

We next consider cases where $G^0_1 = \Unitary(1)$. In such cases, $G_1 \in \{\Unitary(1), N(\Unitary(1))\}$,
so for any given $G$ there exists a group $G_2$ with connected part $G^0_2$ such that either
$G = \Unitary(1) \times G_2$, $G = N(\Unitary(1)) \times G_2$, or $G$ is the fiber product
$N(\Unitary(1)) \times_{\cyc 2} G_2$ with respect to the unique nontrivial homomorphism $N(\Unitary(1)) \to \cyc 2$
and some nontrivial homomorphism $G_2 \to \cyc 2$.
The latter may be characterized by the group $G_2$ and the index-2 subgroup $K$ which is the kernel of the homomorphism.

\begin{itemize}

\item
In the case $G^0 = \Unitary(1) \times \SU(2) \times \SU(2)$  (type $\bF$),
take $G^0_1=\Unitary(1)$ and $G_2^0=\SU(2)\times\SU(2)$; then
$G_2 \in \{\SU(2) \times \SU(2), N(\SU(2) \times \SU(2))\}$
and there is a unique fiber product in the latter case. For this identity component we find 5 extensions, of which one is maximal: $N(\Unitary(1))\times N(\SU(2)\times \SU(2))$.

\item
In the case $G^0 = \Unitary(1) \times \SU(2)_2$ (type $\bJ$), 
$G_2$ must be taken from the groups with $d=3$ in \cite[Table~8]{FKRS12}.
Unique fiber products occur for $G_2 = E_2, E_4, E_6, J(E_1), J(E_3)$.
Multiple fiber products occur for $G_2 = J(E_{2n})$ for $n=1,2,3$,
 as we may take the kernel to be either $E_{2n}$ or one of the two copies of $J(E_n)$. (The latter are conjugate to each other via the normalizer of $J(E_{2n})$, and so give rise to conjugate fiber products.) For this identity component we find 31 extensions, of which two are maximal: $N(\Unitary(1))\times J(E_4)$ and $N(\Unitary(1))\times J(E_6)$.

\item
In the case $G^0 = \Unitary(1) \times \Unitary(1)_2$ (type $\bL$), we find the maximal extensions $N(\Unitary(1))\times J(D_6)$ and $N(\Unitary(1))\times J(O)$. We defer to \cite{FKS} the discussion yielding 122 extensions in this case.
\end{itemize}

\subsubsection{Triple products}\label{section: tripleprodclass}

We next classify the options for $G$ assuming that $G^0 \simeq \SU(2) \times \SU(2) \times \SU(2)$ (type $\bE$).
In this case, $N/G^0 \simeq \sym 3$; more precisely, $N$ is generated by $G^0$ plus the
permutation matrices. We find four options for $G$ in this case, corresponding to the subgroups of $\sym 3$ up to conjugacy,
which we may identify with $\sym 1$, $\sym 2$, $\alt 3$, $\sym 3$. We obtain a unique maximal extension in this case.

We next classify the options for $G$ assuming that $G^0 \simeq \Unitary(1) \times \Unitary(1) \times \Unitary(1)$
(type $\bH$).
In this case, $Z = G^0$ and $N/G^0$ is isomorphic to the wreath product $\cyc 2 \wr \sym 3$. We find 33 options for $G$ in this case, all contained in a unique maximal extension (we will return to this case in \S\ref{section: upperbound}). 

\subsubsection{The diagonal product $\SU(2)_3$}
We next classify the options for~$G$ assuming that $G^0 \simeq \SU(2)_3$ (type $\bM$).
In this case, $Z$ equals the group $\mathrm{O}(3)$ realized as a block matrix with scalar entries,
$Z \cap G^0 = \{\pm 1\}$, 
and $N = ZG^0$, so $N/G^0 \simeq \SO(3)$. We first bound the order of a cyclic subgroup of $\SO(3)$ satisfying (ST3). 
Any cyclic subgroup of $\SO(3)$ of order $n$ is conjugate to
\[
\langle A \rangle, \qquad
A = \begin{pmatrix} 1 & 0 & 0 \\
0 & \cos 2\pi/n & \sin 2\pi/n \\
0 & -\sin 2\pi/n & \cos 2\pi/n
\end{pmatrix}.
\]
By (ST3) the average of the trace of the representation $\wedge^2\C^6$ on the coset of $G^0$ containing $A$, which is given by
$$
\Exp[\Trace(\gamma, \wedge^2 \C^{6}): \gamma \in AG^0]=|\Trace(A)|^ 2=(1 + 2\cos 2 \pi/n)^2,
$$
must be an integer;
this happens if and only if $n \in \{1,2,3,4,6\}$.
Since every finite subgroup of $\SO(3)$ is either cyclic, dihedral, or one of the three exceptional groups
(tetrahedral, octahedral, icosahedral), we obtain component
groups isomorphic to 
\[
\cyc 1,\, \cyc 2,\, \cyc 3,\, \cyc 4,\, \cyc 6,\, \dih 2,\, \dih 3,\, \dih 4,\, \dih 6,\, \alt 4,\, \sym 4.
\]
For this identity component we find 11 extensions, among which two are maximal, namely those with component groups $\dih 6$ and $\sym 4$, which we denote by $\SU(2)_3[\dih 6]$ and $\SU(2)_3[\sym 4]$.

\subsubsection{The diagonal product $\Unitary(1)_3$}
We finally report on the classification for the options for~$G$ achieved in \cite{FKS} assuming that $G^0 \simeq \Unitary(1)_3$ (type $\bN$).
In this case, $Z = \Unitary(3)$ embedded as in \S\ref{section: identity component}, $G^0 \subset Z$, and $N = Z \cup JZ$,
so $N/G^0 \simeq \PSU(3) \rtimes \cyc{2}$ for the action of $\cyc{2}$ on $\PSU(3)$ via complex conjugation.

As in the case $G^0 \simeq \SU(2)_3$, if the cyclic subgroup of $\PSU(3)$ generated by $A \in Z$ satisfies (ST3), then
$\left| \Trace(A) \right|^2$ is an integer (and similarly for every power of $A$). It will suffice to apply (ST3) in this form, which we denote by (ST3$'$), as we will show \emph{a posteriori} in 
\S\ref{section: realization} that every group we obtain also satisfies (ST3) in full.

We first focus on subgroups of $\PSU(3)$. Since $\SU(3)$ surjects onto $\PSU(3)$ with kernel $\mu_3$, we may identify finite subgroups of $\PSU(3)$ with finite subgroups of $\SU(3)$ containing $\mu_3$;
for the latter, we use the classification by Blichfeldt--Dickson--Miller \cite[Chapter~XII]{MBD61}.
\begin{lemma} \label{lemma: BDM}
Every finite subgroup of $\SU(3)$ containing $\mu_3$ is conjugate to a subgroup which is either
\begin{itemize}
\item
contained in the diagonal torus,
\item
contained in $\SU(3) \cap (\Unitary(1) \times \Unitary(2))$,
\item
a semidirect product of $\alt 3$ acting on a diagonal abelian group,
\item
a (not necessarily split) extension of $\sym 3$ by a diagonal abelian group, or
\item
one of seven exceptional cases whose images in $\PSU(3)$ have orders 
$36$, $72$, $216$ (the Hessian group), $60$ (the group $\alt 5$),
$360$ (the group $\alt 6$), or $168$ (the group $\PSL(2, 7)$).
\end{itemize}
\end{lemma}

To reduce the classification to a finite problem, we study the abelian case in detail.
In what follows, let $D(u,v,w)$ denote the $3 \times 3$ diagonal matrix with diagonal entries $e^{2 \pi i u}, e^{2 \pi i v}, e^{2 \pi i w}$.

\begin{lemma} \label{lemma: type N cyclic subgroup order}
Let $H$ be a finite cyclic subgroup of $\SU(3)$ containing $\mu_3$ and satisfying (ST3$'$).
Then $\#(H/\mu_3)$ divides one of $7, 8, 12$. 
\end{lemma}
\begin{proof}
By conjugating, we may ensure that $H$ is generated by some matrix of the form $D(u,v,w)$ with $u,v,w \in \Q$
and $u+v+w \in \Z$,
and the problem is to bound the least common denominator of $u,v,w$ subject to the condition that
\[
e^{2\pi i (u-v)} + 
e^{2\pi i (u-w)} + 
e^{2\pi i (v-u)} + 
e^{2\pi i (v-w)} + 
e^{2\pi i (w-u)} + 
e^{2\pi i (w-v)}  \in \Z.
\]
This amounts to finding the cyclotomic points on a certain algebraic curve
(i.e., solving a multiplicative Manin-Mumford problem), which can be done by 
following either Conway--Jones \cite{CJ76} or Beukers--Smyth \cite{BS02}; we defer to \cite{FKS} for further details.
\end{proof}

\begin{remark} \label{remark: type N order divisible by 7}
By Lemma~\ref{lemma: type N cyclic subgroup order}, the only primes that can divide the order of a finite subgroup of $\PSU(3)$ satisfying (ST3$'$) are $2,3,7$; in particular, the exceptional subgroups $\alt 5, \alt 6$ are excluded. We may moreover classify all of the cases in which~$7$ divides the order as follows.

Up to conjugation, there is a unique copy $H_0$ of $\cyc 7$ in $\PSU(3)$ satisfying (ST3$'$); it is generated by the image of $D(\tfrac{1}{7}, \tfrac{2}{7}, \tfrac{4}{7})$. 
By Lemma~\ref{lemma: type N cyclic subgroup order}, $H_0$ is not contained in any larger abelian group satisfying (ST3$'$); it is also not contained in any conjugate of $\SU(2)$. 
By Lemma~\ref{lemma: BDM} plus a direct verification that $\PSL(2,7)$ satisfies (ST3$'$), it follows (after checking that these cases do all work) that up to conjugacy, the finite subgroups of $\PSU(3)$ satisfying (ST3$'$) of order divisible by 7 are
\[
H_0,\ H_0 \rtimes \alt 3,\ \PSL(2, 7).
\]
Having addressed the subgroups of $\PSU(3)$ of order divisible by $7$, we henceforth restrict our attention to finite subgroups of $\PSU(3)$ with 3-smooth order.
\end{remark}

\begin{remark} \label{remark: type N abelian cases}
Let $H$ be a finite subgroup of $\SU(3)$ of 3-smooth order containing $\mu_3$ and satisfying (ST3$'$).
If $H$ is a diagonal abelian group, then it must be isomorphic to $\cyc{m} \times \cyc{n}$ for some integers $m,n$ with $\lcm(3,m)$ dividing $n$. It is straightforward to exhaustively consider all choices of $m,n$ consistent with Lemma~\ref{lemma: type N cyclic subgroup order} to identify cases satisfying (ST3$'$); we find that in all cases,
$m \in \{1,2,3,4,6\}$.

If $H$ is an extension of $\alt 3$ by a diagonal abelian group $H_0$, then $H_0$ cannot be cyclic (because no cyclic group of order dividing $8$ or $12$ has a faithful automorphism of order $3$). It must thus have the form
$\cyc{m} \times \cyc{n}$ with $m \in \{2,3,4,6\}$. It is straightforward to identify the cases from the previous
paragraph which are stable under $\alt 3$; each of these cases is in fact stable under $\sym 3$.
\end{remark}

\begin{remark}
Using Lemma~\ref{lemma: type N cyclic subgroup order}, Remark~\ref{remark: type N order divisible by 7}, and
Remark~\ref{remark: type N abelian cases},
the classification of finite subgroups of $\SU(3)$ containing $\mu_3$ and satisfying (ST3$'$) 
reduces to the case of a nonabelian subgroup $H$ of $\SU(3) \cap (\Unitary(1) \times \Unitary(2))$.
The latter group admits $\Unitary(1) \times \SU(2)$ as a double cover via the map $(u, A) \mapsto (u^2, u^{-1} A)$.
For $\tilde{H}$ the preimage of $H$ in this cover, the projection of $\tilde{H}$ to $\SU(2)$ cannot be an abelian group of odd order; it must therefore be the inverse image of a dihedral, tetrahedral, or octahedral subgroup of $\SO(3)$. 
Meanwhile, the projection of $\tilde{H}$ to $\Unitary(1)$ has order dividing 16 or 24.
One may thus obtain a bound on such groups by testing (ST3$'$) for each finite subgroup (up to conjugacy) of each of the subgroups
\[
\cyc{96} \times 2\dih{48},\
\cyc{96} \times 2O.
\]
We were not able to completely automate this calculation; the classification in \cite{FKS} is human-readable, although we did cross-check several steps by numerical computation.
\end{remark}

It remains to classify finite subgroups of $\PSU(3) \rtimes \cyc{2}$ not contained in $\PSU(3)$
whose intersection with $\PSU(3)$ equals a fixed group $H$ known to satisfy (ST3$'$); we call each such subgroup a 
\emph{double cover} of $H$.
(In general, not all double covers are semidirect products of $H$ by $\cyc{2}$, but the maximal groups all are.)
In most cases, we work with a conjugacy representative of $H$ 
which is stable under the action $g \mapsto \overline{g}$ of complex conjugation;
under this condition, the double covers of $H$ are the groups of the form $\langle H, Jg \rangle$ for some $g \in N_H$ satisfying $\overline{g} g \in H$. 
The last condition means that the image of $g$ in $N_H/H$ maps to its inverse under complex conjugation; the
values in $N_H/H$ corresponding to the same double cover as $g$ are the images of $\overline{h}^{-1} g h$ for all $h \in N_H$.

\begin{remark}
One complicating feature is that not every double cover of a given $H$ extends to every overgroup. In particular,
the intersection of a maximal finite subgroup of $\PSU(3) \rtimes \cyc{2}$ satisfying (ST3$'$) with $\PSU(3)$ need not be a \emph{maximal} finite subgroup of $\PSU(3)$ satisfying (ST3$'$).
\end{remark}

\begin{theorem} \label{theorem: type N maximal subgroups}
Up to conjugacy, there are $12$ maximal subgroups of $\PSU(3) \rtimes \cyc{2}$ satisfying (ST3$'$),
as listed in Table~\ref{table: maximal subgroups U1}. (See below for explanation of the table entries.)
\end{theorem}
The maximal subgroups in Theorem~\ref{theorem: type N maximal subgroups} are described in  Table~\ref{table: maximal subgroups U1}. In each row, $M$ is an imaginary quadratic field of class number $1$,
$\tilde{H}$ is a subgroup of $\GL_3(M)$, $H$ is the projective image of $\tilde{H}$,
and $H \rtimes \cyc 2$ is a particular choice of semidirect product (to be discussed in more detail in
Lemma~\ref{lemma: complex reflection}). The notation for $\tilde{H}$ is to be interpreted as follows.
\begin{itemize}
\item
The labels $G(m,n,p)$ and $G_i$ refer to the Shephard-Todd notation \cite{ST54} for complex reflection groups.
In these cases, $M$ is the field of traces of $\tilde{H}$ as per \cite[Table~1]{MM10}.
\item
The notation $[G_{8}:4]$ refers to a non-normal index-4 subgroup of $G(1,1,1) \times G_8$. It is not a complex reflection group.
\item
The notation $G'_4$ refers to a copy of $G_4$ embedded via its two-dimensional irreducible representation with rational traces (a cubic twist of the reflection representation). That representation has a nontrivial Schur index: it can only be realized over fields that split the quaternion algebra $(-1,-1)_{\Q}$ (compare \cite[Proposition~3.5]{FG18}). By convention, we do not regard $G'_4$ as a complex reflection group.
\end{itemize}

The action of $\cyc 2$ on $H$ is induced by an action of $\cyc 2$ on $\tilde{H}$ specified as follows.
\begin{itemize}
\item
For $c \in \cyc 2$ and $c_M \in \Gal(M/\Q)$ generators and $\rho: \tilde{H} \to \GL_3(M)$ the implied representation, we have $\rho \cong c_M \circ \rho \circ c$.
\item
In the cases where $\tilde{H}$ is written as a product, the action of $c$ respects this decomposition.
\item
The actions on $G(n, 1, 1), G(4,4,3), G(6,2,3), G_4, G_8, G_{12}, G_{24}, G_{25}$ are all outer, represented by the action of $c_M$ on a $c_M$-stable presentation \cite[Theorem~1.7]{MM10}. 
\item
The action on $[G_{8}:4]$ is outer, but this leaves two possibilities for the semidirect product. One of these extends to $G_8$; we take the other one.
\item
The action on $G(3,3,2)$ is trivial.
\item
The actions on $G(2,1,2), G(6,6,2)$ are outer.
\item
The action on $G_4'$ is outer, but this leaves two possibilities for the semidirect product. One of these recovers $\langle 144, 125 \rangle$;
we take the other one.
\end{itemize}

\begin{table}[ht]
\small
\begin{tabular}{|c|c|c|c|c|}\hline 
$H$ & $H \rtimes \cyc 2$ & $\tilde{H}$ & $\#\tilde{H}/\# H$ & $M$ \\
\hline
$\langle 24, 1 \rangle$ & $\langle 48, 15 \rangle$& $G(1,1,1) \times [G_{8}:4]$ & $1$ & $\Q(i)$  \\
$\langle 24,10 \rangle$ & $\langle 48, 15 \rangle$ & $G(3,1,1) \times G(2,1,2)$ & $1$ & $\Q(\zeta_3)$ \\
$\langle 24,5 \rangle$ & $\langle 48, 38 \rangle$ & $G(4,1,1) \times G(3,3,2)$ & $1$ & $\Q(i)$ \\
$\langle 24,5 \rangle$ & $\langle 48, 41 \rangle$ &  $G(4,1,1) \times G(6,6,2)$ & $2$ & $\Q(i)$ \\
$\langle 48, 29 \rangle$  & $\langle 96, 193 \rangle$ &  $G(1,1,1) \times G_{12}$ & $1$ & $\Q(\sqrt{-2})$\\
$\langle 72, 25 \rangle$ & $\langle 144, 125 \rangle$ &  $G(3,1,1) \times G_{4}$ & $1$ & $\Q(\zeta_3)$ \\
$\langle 72, 25 \rangle$  & $\langle 144, 127 \rangle$ &  $G(3,1,1) \times G_4'$ & $1$ & $\Q(\zeta_3)$ \\
$\langle 96, 67 \rangle$   & $\langle 192, 988 \rangle$ &  $G(1,1,1) \times G_{8}$ & $1$ & $\Q(i)$ \\
$\langle 96, 64 \rangle$  &  $\langle 192, 956 \rangle$ & $G(4,4,3)$ & $1$ &$\Q(i)$\\
$\langle 168, 42 \rangle$  &  $\langle 336, 208 \rangle$ & $G_{24}$& $2$ & $\Q(\sqrt{-7})$\\
$\langle 216, 92 \rangle$  & $\langle 432, 523 \rangle$ &  $G(6,2,3)$ & $3$ & $\Q(\zeta_3)$\\
$\langle 216, 153 \rangle$  & $\langle 432, 734 \rangle$ &  $G_{25}$ & $3$ & $\Q(\zeta_3)$\\\hline
\end{tabular}
\bigskip

\caption{Maximal extensions of $\Unitary(1)_3$.
}
\label{table: maximal subgroups U1}
\end{table}

\begin{remark}
Since we are only interested in $H$ rather than $\tilde{H}$,
some variation in the choice of complex reflection groups is possible in Table~\ref{table: maximal subgroups U1}.
For the most part, we have made choices to minimize $\#\tilde{H}/\#H$.

%\begin{itemize}
%\item
%One can replace $G(1,1,1) \times [G_{8}:4]$ with 
%$G(2,1,1) \times [G_{8}:4]$ ($\#\tilde{H}/\#H = 2$) or $G(4,1,1) \times [G_{8}:4]$ ($\#\tilde{H}/\#H = 4$).
%\item
%One can replace $G(3,1,1) \times G(2,1,2)$ with $G(6,1,1) \times G(2,1,2)$ ($\#\tilde{H}/\#H = 2$).
%\item
%One can replace $G(4,1,1) \times G(1,1,3)$ with $G(4,1,1) \times G(3,3,2)$ ($\#\tilde{H}/\#H = 1$)
%or $G(4,1,1) \times G(6,6,2)$ ($\#\tilde{H}/\#H = 2$). 
%However, this is not true in reverse: the outer action of $\cyc 2$ on $G(6,6,2)$ does not preserve $G(1,1,3)$.
%\item
%One can replace $G(1,1,1) \times G_{12}$ with $G(2,1,1) \times G_{12}$ ($\#\tilde{H}/\#H = 2$).
%\item
%One can replace $G(3,1,1) \times G_4$ with $G(6,1,1) \times G_4$ ($\#\tilde{H}/\#H = 2$),
%$G(1,1,1) \times G_5$ ($\#\tilde{H}/\#H = 1$),
%$G(2,1,1) \times G_5$ ($\#\tilde{H}/\#H = 2$),
%$G(3,1,1) \times G_5$ ($\#\tilde{H}/\#H = 3$),
%or
%$G(6,1,1) \times G_5$ ($\#\tilde{H}/\#H = 6$),
%\item
%One can replace $G(1,1,1) \times G_8$ with $G(2,1,1) \times G_8$ ($\#\tilde{H}/\#H = 2$)
%or $G(4,1,1) \times G_8$ ($\#\tilde{H}/\#H = 4$).
%\item
%One can replace $G(4,1,3)$ with $G(4,2,3)$ ($\#\tilde{H}/\#H = 2$) or $G(4,4,3)$ ($\#\tilde{H}/\#H = 1$).
%\item
%One can replace $G(6,1,3)$ with $G(6,2,3)$ ($\#\tilde{H}/\#H = 3$).
%\item
%One can replace $G_{25}$ with $G_{26}$ ($\#\tilde{H}/\#H = 6$).
%\end{itemize}
\end{remark}

\begin{remark}
By choosing explicit presentations of the groups in Table~\ref{table: maximal subgroups U1}, it is possible to verify
by computer calculations that the subgroups of these groups account for 171 distinct extensions of $\Unitary(1)_3$,
and that there are no containment relations among the groups in the table. The difficult part of the classification, which we defer to \cite{FKS}, is to check that every possible extension is accounted for in this manner; in other words, that
there are no maximal subgroups missing from the table below.
\end{remark}

\begin{table}[h]
\small
\[
\begin{array}{|c|c|c|c|}\hline
\ST(A)^0  &\dim \ST(A)^0 & \mbox{Extensions} & \mbox{Maximal}\\
\hline
\USp(6)& 21 & 1 & 1\\\hline
\Unitary(3) &9 & 2 & 1\\\hline
\SU(2) \times \USp(4) &13 & 1 & 1\\\hline
\Unitary(1) \times \USp(4) &11 & 2 & 1\\\hline
\SU(2) \times \SU(2) \times \SU(2) &9 & 4 & 1\\\hline
\Unitary(1) \times \SU(2) \times \SU(2) &7 & 5 & 1\\\hline
\Unitary(1) \times \Unitary(1)\times \SU(2) &5 & 8 & 1\\\hline
\Unitary(1) \times \Unitary(1) \times \Unitary(1) &3 & 33 & 1\\\hline
\SU(2) \times \SU(2)_2 & 6 & 10 & 2\\\hline
\Unitary(1) \times \SU(2)_2 & 4 &31 & 2\\\hline
\SU(2) \times \Unitary(1)_2 & 4 & 32 & 2\\\hline
\Unitary(1) \times \Unitary(1)_2 & 2 &122 & 2\\\hline
\SU(2)_3 & 3 & 11 & 2\\\hline
\Unitary(1)_3 & 1 & 171 & 12\\
\hline

\end{array}
\]
\vspace{6pt}
\caption{For each connected Sato--Tate group, we list the number of extensions satisfying the Sato--Tate axioms and the number of those which are maximal.
}\label{table: extensions}
\end{table} 

\section{Realization}\label{section: realization}

The goal of this section is to provide a proof of the following result, assuming Theorem \ref{theorem: upper bound}.

\begin{theorem} \label{theorem: lower bound}
Up to conjugacy, there are $410$ closed subgroups of $\USp(6)$ that occur as Sato--Tate groups of abelian threefolds over number fields. 
Among these, the $33$ that are maximal with respect to finite index inclusions arise as Sato--Tate groups of abelian threefolds.
(Hence by base extension, all $410$ groups occur for abelian threefolds over number fields.)
\end{theorem}

The proof of the theorem has two parts. On the one hand, when the identity component contains
$\Unitary(1) \times \Unitary(1)$ as a factor, we refine the group-theoretic argument by making a further analysis of CM types that rules out 23 of the groups satisfying the Sato--Tate axioms (\S\ref{section: upperbound}). In particular, the set of maximal groups among those satisfying the refined analysis of CM types differs from the set of maximal groups satisfying the Sato--Tate axioms (see Remark \ref{remark: maximalgroups}). 

On the other hand, we realize the remaining 410 groups as Sato--Tate groups of certain abelian threefolds (\S\ref{section: lowerbound}). Crucially, it will suffice to realize the 33 maximal groups among these, as then their subgroups of finite index can be achieved by considering base extensions (by Lemma~\ref{lemma: Galois action of component group}). As an aside, we note that at present we do not know exactly which groups can occur over $\Q$ (certainly not all of them: no connected group with a unique factor isomorphic to $\Unitary(1)$ can arise over~$\Q$, since this would require a CM elliptic curve to have all its endomorphisms defined over~$\Q$).

We note in passing that all of the examples of types other than $\bN$ will carry a principal polarization. 
The analysis of polarizations in the type $\bN$ case is more subtle, and is not included here.

\subsection{Upper bound}\label{section: upperbound}
Let $A$ be an abelian threefold with $\ST(A)=G$ and suppose that $G^0$ contains $\Unitary(1)\times\Unitary(1)$ as a factor.

Suppose first that $G^0= \SU(2)\times \Unitary(1)\times\Unitary(1)$. Then $A$ must split up to isogeny as the product of an elliptic curve without CM and an abelian surface
whose Sato--Tate group has identity component $\Unitary(1) \times \Unitary(1)$. By \cite[\S4.3, \S4.4]{FKRS12}, of the 8 options
for $G$ described in \S\ref{section: splitproductsclass}, we may omit the 3 for which $G_2$ cannot occur for abelian surfaces.

Suppose next that $G^0= \Unitary(1)\times \Unitary(1) \times \Unitary(1)$ so that $N/G^0$ is isomorphic to the wreath product $\cyc 2 \wr \sym 3$. Let $a,b,c$ be representatives of the nontrivial cosets in the three factors of $N(\Unitary(1))$,
and set $t,s$ correspond to the permutations $(12), (123)$ in~$\sym 3$.

If $A$ is not simple, then it splits as a product of a CM elliptic curve and a CM abelian surface;
this means that up to conjugacy, $\ST(A)$ is contained in the product of a subgroup of 
$\langle a,b,t \rangle$ and a subgroup of $\langle c \rangle$. 
By the analysis from \cite[\S 4.3]{FKRS12}, the former must be a subgroup of either $\langle a,b \rangle$
or $\langle at \rangle$.

To handle the case where $A$ is simple, 
we recall some basic facts about abelian varieties with complex multiplication,
following \cite[\S 1]{Dod84} and \cite[\S 1]{Kil16} (and ultimately \cite[\S 5.5]{Shi71}).
For future reference, we temporarily allow $A$ to have arbitrary dimension.

\begin{definition}
A \emph{CM field} is a totally imaginary quadratic extension of a totally real number field.
A \emph{CM type} is a pair $(M, \Phi)$ consisting of a CM field $M$ and a section $\Phi$ of the map grouping the complex embeddings of $M$ into conjugate pairs.

Given a CM type $(M, \Phi)$, for any $\Z$-lattice $\mathfrak{m}$ in $M$, we may use $\Phi$ to specify a complex structure on $M \otimes_\Q \R$, and thus view $(M \otimes_\Q \R)/\mathfrak{m}$ as a complex torus. Changing the choice of $\mathfrak{m}$ gives rise to an isogenous torus \cite[Proposition~5.13]{Shi71}; consequently, we may associate a polarization to this torus using the trace pairing on $M$, and thus obtain an abelian variety $A$ over $\C$ with complex multiplication by some order in 
$M$. Note that if $M$ has degree $2n$, then $A$ has dimension $n$.
\end{definition}

\begin{definition}
Given a CM type $(M, \Phi)$, let $L$ be the Galois closure of $M$ in $\C$, which is again a CM field.
The action of $\Gal(L/\Q)$ on the complex embeddings of $M$ by postcomposition acts on the CM types for $M$.
The \emph{reflex field} of $(M, \Phi)$ is the fixed field $M^*$ of the stabilizer of $\Phi$ in $\Gal(L/\Q)$; note that $M^*$ is defined as a subfield of $L$ (and hence of $\C$), 
whereas $M$ does not come with a distinguished embedding into $L$.
(There is also a CM type associated to $M^*$, the \emph{reflex type}, which we will not discuss here.)
\end{definition}

\begin{lemma}
Fix an embedding of $k$ into $\C$. 
Let $A$ be a simple polarized abelian variety of dimension $n$ over $k$ with CM type $(M, \Phi)$.
\begin{enumerate}
\item[(a)]
The field $M^*$ is Galois over $k \cap M^*$.
\item[(b)]
The field $kM^*$ is the endomorphism field of $A$.
\item[(c)]
The group $\Gal(kM^*/k)$ is canonically isomorphic to a subgroup of $\Aut(M/\Q)$.
\end{enumerate}
\end{lemma}
\begin{proof}
For (a), (b), see \cite[Theorem~5.15, Proposition~5.17]{Shi71}. 
For (c), the action from (b) gives an action of $\Gal(kM^*/k)$ on $\End(A_\C)_{\Q} = M$.
\end{proof}

\begin{remark}
To trace the effect of the previous discussion on Sato--Tate groups, we identify $\ST(A)/\ST(A)^0$
with a subgroup of $\cyc 2 \wr \sym n$ using the embedding of $M$ into $\End(A_K)_{\R}$
to fix the identification of $\Unitary(1) \times \cdots \times \Unitary(1)$ with $\prod_{\tau \in \Phi} \C^\times$.
Then the subgroup $\ST(A)/\ST(A)^0$ is contained in $\Gal(L/(k \cap M^*))$
and the composition
\[
\ST(A)/\ST(A)^0 \to \Gal(L/(k \cap M^*)) \to \Gal(M^*/(k \cap M^*)) \cong \Gal(kM^*/k)
\]
is the canonical map (see Remark~\ref{remark: general ST group}).
\end{remark}

We now specialize back to the case where $A$ is 
a simple CM abelian threefold. In this case, as per \cite[\S 5.1.1]{Dod84},
the sextic field $M$ must be one of the following.
\begin{itemize}
\item
A cyclic Galois extension of $\Q$. In this case,
we must have $G/G^0 \subseteq \langle abc,s \rangle$ as this is the unique copy of $\cyc 6$ in $\cyc 2 \wr \sym 3$.

\item
A compositum of an imaginary quadratic field $M_0$ and a non-Galois totally real cubic field $M_1$.
In this case, $\Aut(M/\Q) = \Gal(M_0/\Q)$ and so $G/G^0 \subseteq \langle abc \rangle$.

\item
A non-Galois extension of $\Q$ whose Galois closure has group $\cyc 2 \wr \alt 3$ or $\cyc 2 \wr \sym 3$.
In this case, $\Aut(M/\Q)$ is trivial, as then is $G/G^0$.

\end{itemize}

To summarize, we list the 33 conjugacy classes of subgroups of $N/G^0$ (as may be verified using \Gap\ or \Magma) in Table~\ref{table:subgroups SU2xSU2xSU2}. In the table, $\star$ indicates a normal subgroup, while an underline denotes a group that is permitted by the previous analysis to occur as a Sato--Tate group of an abelian threefold; we will see later that these groups all do occur (see \S\ref{section: tripleprodreal}).

\begin{table}
\small
\begin{tabular}{|c|c|c|c|}\hline
Order & Isom type &Groups & Total \\
\hline
1 & $\cyc 1$ & $\underline{\langle e \rangle}$ & 1 \\
2 & $\cyc 2$ & $\underline{\langle a \rangle}$, $\underline{\langle ab \rangle}$, $\underline{\langle abc \rangle}\star$, $\langle t \rangle$, $\langle ct \rangle$& 5 \\
3 & $\cyc 3$ & $\underline{\langle s \rangle}$ & 1 \\
4 & $\cyc 4$ & $\underline{\langle at \rangle}$, $\underline{\langle act \rangle}$ &2 \\
4 & $\cycsup{2}{2}$ &  $\underline{\langle a,b \rangle}$, $\underline{\langle a,bc \rangle}$, $\underline{\langle ab, bc \rangle}\star$, $\langle c,t\rangle$,
$\langle ab,t \rangle$, $\langle ab, ct \rangle$, $\langle abc, t \rangle$ &7 \\
6 & $\cyc 6$  & $\underline{\langle abc,s \rangle}$&1 \\
6 & $\sym 3 $  & $\langle s,t \rangle$, $\langle abct, s \rangle$ &2 \\
8 & $\cyc{2} \times \cyc{4}$  & $\underline{\langle c,at \rangle}$ &1 \\
8 & $\cycsup{2}{3}$  & $\underline{\langle a,b,c \rangle}\star$,$\langle ab,c,t \rangle$&2 \\
8 & $\dih 4$  & $\langle a,b,t \rangle$, $\langle ab,bc,t \rangle$,
$\langle a,b,ct \rangle$, $\langle ab,bc,ct \rangle$ &4\\
12 & $\dih 6$ & $\langle abc,s,t \rangle$ &1\\
12 & $\alt 4$  & $\langle ab,bc,s \rangle\star$ &1\\
16 & $\cyc{2} \times \dih 4$ & $\langle a,b,c,t \rangle$&1\\
24 & $\cyc 2 \wr \alt 3$ &$\langle a,b,c,s \rangle\star$&1\\
24 & $\sym 4$ & $\langle ab,bc,s,t \rangle\star$, 
$\langle ab,bc,at,s \rangle\star$&2\\
48 & $\cyc 2 \wr \sym 3$ & $\langle a,b,c,s,t \rangle\star$ &1\\ \hline
\end{tabular}
\vspace{6pt}
\caption{Conjugacy classes of subgroups of $N/G^0$ for $G^0 \simeq \Unitary(1) \times \Unitary(1) \times \Unitary(1)$. The starred groups are normal; the underlined groups can be realized by Sato--Tate groups.}
\label{table:subgroups SU2xSU2xSU2}
\end{table}

\begin{remark}\label{remark: maximalgroups} The above discussion eliminates, in particular, the groups $\SU(2)\times F_{a,b,c}$ and $N(\Unitary
(1)\times \Unitary(1)\times \Unitary(1))$. These are the maximal groups among those satisfying the Sato--Tate axioms for the absolute types $\bG$ and $\bH$ as seen in \S\ref{section: splitproductsclass} and \S\ref{section: tripleprodclass}, respectively. In their place, five groups become maximal among the list of groups satisfying the above refined analysis. The groups $\SU(2)\times F_{a,b}$ and $\SU(2)\times F_{ac}$ become maximal for the absolute type $\bG$, and the groups with groups of connected components $\langle a,b,c \rangle$, $\langle at,c \rangle$, and  $\langle abc,s \rangle$ become maximal for the absolute type $\bH$. The refined analysis thus yields 33 maximal groups in total. 
\end{remark}

\begin{remark}\label{remark: lowerbound}
Recall from Remark~\ref{remark: upper bound} that Theorem~\ref{theorem: upper bound} gives only an upper bound on the set of conjugacy classes of subgroups of $\GL_6(\C)$ satisfying the Sato--Tate axioms for an abelian threefold.
The lower bound is not obtained in this argument because we only verify (ST3) for some choices of the character $\chi$.

However, it can be shown that all 433 conjugacy classes that occur do in fact satisfy (ST3) in full. For the 410 groups that occur in Theorem~\ref{theorem: lower bound}, the realizations of these groups as Sato--Tate groups of abelian threefolds will \emph{a posteriori} confirm (ST3). 
To settle the remaining cases, it suffices to check (ST3) for the full normalizers of
$\Unitary(1) \times \Unitary(1) \times \SU(2)$
and $\Unitary(1) \times \Unitary(1) \times \Unitary(1)$ in $\USp(6)$.
For $G$ equal to one of these groups, we can write $G$ as $G_0 \rtimes K$ where $K$ is a finite subgroup of $\USp(6) \cap \GL_6(\Q)$; this means that (ST3) holds for every $\chi$ which is the base extension of a $\Q$-linear representation of $\GL_6(\Q)$, and this covers all $\chi$ by virtue of the usual description of irreducible representations of $\GL_6(\C)$ via Schur functors.
\end{remark}

\subsection{Lower bound}\label{section: lowerbound} We next exhibit examples of abelian threefolds defined over $\Q$ realizing each of the 33 maximal groups described in Remark \ref{remark: maximalgroups}.  

\subsubsection{Indecomposable cases} The unique group of absolute type $\bA$ is realized by any abelian threefold with trivial endomorphism ring,
and in particular by the Jacobian of a generic genus-$3$ curve.
For a concrete example of absolute type $\bA$, by a theorem of Zarhin \cite[Theorem~2.1]{Za00}, the Jacobian of the hyperelliptic curve $y^2 = x^7 - x + 1$ has trivial endomorphism ring.

The unique maximal group of absolute type $\bB$ is realized by the Jacobian of a generic Picard curve. 
For a concrete example of absolute type $\bB$, we take an example of Upton \cite[\S 6.1.3]{Upt09}: the Jacobian of the Picard curve $C: y^3 = x^4 + x + 1$ has maximal $67$-adic Galois image. Note that Upton's result uses a calculation of Frobenius traces plus a big image criterion in the manner of \cite{Vas04}; in general, we do not know of a good way to compute directly with the mod-$\ell$ Galois representation for any prime $\ell$ other than $2$, for which the Galois action on bitangents provides a computable model.

\subsubsection{Split products} Note that the maximal split products are all of the form $G=G_1\times G_2$, where $G_1=N(\Unitary(1))$ or $\SU(2)$, and 
$$
G_2=\USp(4),\, N(\SU(2)\times \SU(2)),\, F_{ac},\,F_{a,b},\, J(E_4),\,J(E_6),\,J(D_6),\,\text{ or } J(O).
$$ 
Fix one such $G_2$. As shown by \cite[Table 11]{FKRS12}, there exists a principally polarized abelian surface~$S$ defined over $\Q$ with $\ST(S)\simeq G_2$. If $G_1=\SU(2)$, then $G$ is realized by the product of $S$ with an elliptic curve without CM defined over $\Q$; if $G_1=N(\Unitary(1))$, then $G$ is realized by the product of $S$ with an elliptic curve defined over $\Q$ with CM by a quadratic imaginary field not contained in the endomorphism field of $S$. One readily verifies that examples of elliptic curves satisfying this condition can be easily found.

\subsubsection{Triple products}\label{section: tripleprodreal} The unique maximal group for the absolute type $\bE$ can be realized in the following manner. Let $E$ be an elliptic curve without CM defined over a non Galois cubic extension $F/\Q$ which is not a $\Q$-curve (that is, there exists $\sigma \in G_\Q$ such that $E$ and $\acc{\sigma}E$ are not $\Qbar$-isogenous). It immediately follows that if $\sigma$ and $\sigma'$ are different embeddings of $F$ into $\Qbar$, then $\acc\sigma E$ and $\acc{\sigma'}E$ are not $\Qbar$-isogenous, and therefore the maximal group for the absolute type $\bE$ is realized by taking the Weil restriction of scalars of $E$ from $F$ down to $\Q$.

We next realize the maximal groups of absolute type $\bH$:
\begin{itemize}
\item The product of three elliptic curves defined over $\Q$ with CM by pairwise non-isomorphic imaginary quadratic fields realizes the maximal group with group of components $\langle a,b,c\rangle$.
\item Let $S$ be a principally polarized abelian surface over $\Q$ with CM by a quartic CM field (as the example given in \cite[Table 11]{FKRS12} realizing the group $F_{ac}$, for example). Let $E$ be an elliptic curve defined over $\Q$ with CM by a quadratic imaginary field not contained in the endomorphism field of $S$. Then $E\times S$ realizes the maximal group with group of components $\langle at,c\rangle$.
\item As discussed in \S\ref{section: upperbound}, the maximal group with group of components $\langle abc,s\rangle$ is realized by a principally polarized abelian threefold with CM by a cyclic CM field of degree 6. As a concrete example, we may take the Jacobian of the curve $y^2=x^7+1$, which has CM by the field $\Q(\zeta_7)$.
\end{itemize}

\subsubsection{The diagonal product $\SU(2)_3$.}\label{section: diagonalprodreal} 
We next realize the two maximal groups of absolute type $\bM$. We will use the construction of Example \ref{example: twistcube}, but this does not yield examples carrying a principal polarization; we thus supplement the discussion with some explicit examples that are principally polarized.

In order to realize $\SU(2)_3[\sym 4]$, let $L/\Q$ denote a Galois extension with Galois group $\sym 4$ and let 
$$
\xi\colon \Gal(L/\Q)\rightarrow \GL_3(\Z)
$$
denote a faithful integral degree $3$ representation. Let $E$ be an elliptic curve without complex multiplication defined over $\Q$ and let $A$ be the twist of $E^3$ by $\xi$. From the fact that $\xi$ has projective image $\sym 4$, we easily deduce that $A$ has Sato--Tate group $\SU(2)_3[\sym 4]$.

Alternatively, note that this group occurs for the Jacobian of a generic twist of the curve $y^2 = x^8 + 14x^4 + 1$, which has an automorphism group of order 48.
See for example \cite[Table~4, row~11]{ACLLM18}.

In order to realize $\SU(2)_3[\dih 6]$, let $L/\Q$ denote a Galois extension with Galois group $\dih 6$ and let 
$$
\xi\colon \Gal(L/\Q)\rightarrow \GL_2(\Z)
$$
denote a faithful integral degree $2$ representation. Let $E$ be an elliptic curve without complex multiplication defined over $\Q$, set $S=E^2$, and let $S_\xi$ be the twist of $S$ by $\xi$. Since the projective image of $\xi$ is $\dih 3$, the endomorphism field of $S_\xi$ has Galois group $\dih 3$. Then $\SU(2)_3[\dih 6]$ is realized by the product of $S_\xi$ and a nontrivial quadratic twist $E'$ of $E$ over a quadratic field not contained in the endomorphism field of $S$.

Alternatively, let $C_2$ be the curve $y^2 = x^6 + x^3 + 4$, which is given in \cite[Table~11]{FKRS12} as an example with Sato-Tate group $J(E_3)$. The curve $C_2$ is a twist of $y^2 = x^6 + x^3/2 + 1$, whose
quotient by the involution
\[
(x,y) \mapsto (1/x, y/x^3)
\]
may be identified with the curve $C_1: y^2 = (x+2)(x^3 - 3x + 1/2)$ via the map
$(x,y) \mapsto (x+1/x, y(x+1)/x^2)$; taking the product of $\Jac(C_2)$ with a generic quadratic twist of $\Jac(C_1)$ yields an abelian threefold with Sato-Tate group $\SU(2)_3[\dih 6]$.

\begin{remark}
Note that in both of the constructions realizing $\SU(2)_3[\dih 6]$, we use an abelian surface with Sato-Tate group $J(E_3)$. One may wonder whether it is possible to construct an abelian threefold $A$ defined over $\Q$ with $\ST(A)=\SU(2)_3[\dih 6]$ such that $A$ is the product of an elliptic curve $E$ over $\Q$ and an abelian surface $S$ over $\Q$ with $\ST(S)=J(E_6)$. This seemingly natural construction is ruled out by \cite[\S3D]{FS14}, which shows that an elliptic factor of such an $S$ never admits a model up to isogeny defined over~$\Q$.
\end{remark}

\subsubsection{The diagonal product $\Unitary(1)_3$}\label{section: diagonalprodimag} 

We finally realize the twelve maximal groups of absolute type $\bN$. We do this in a uniform fashion using the presentations given in Table~\ref{table: maximal subgroups U1}. Hereafter, let $H, \tilde{H}$ correspond to a row of the table.

We first verify that the construction of 
Example~\ref{example: twistcube CM} is applicable.
In the cases where $\tilde{H}$ is written as a product, we apply Remark~\ref{remark: twisting and principal polarizations} to the two-dimensional factor. Using the results of \cite{Fei03}, we see the following.
\begin{itemize}
\item
For the factors $G_{12}, G_8$, the integral representation $\rho_0$ is uniquely determined by $\rho$,
so point (i) of Remark~\ref{remark: twisting and principal polarizations} applies.
\item
For the factor $G(3,3,2)$, point (ii) of Remark~\ref{remark: twisting and principal polarizations} applies.
\end{itemize}

To handle the remaining factors, take $\alpha = 3$ if $M = \Q(i)$ or $\alpha=2$ if $M = \Q(\zeta_3)$.
In Table~\ref{table:two-dim generators}, we give explicit generators for the image of one choice of $\rho_0$,
presented so that the reduction modulo $\alpha$ becomes lower triangular.
One may check easily that while the reduction modulo $\alpha$ is manifestly reducible, the reduction modulo $\alpha^2$ admits no invariant vector not divisible by $\alpha$.
We may thus apply \cite[Corollary~1.5]{Fei03} to see that there are exactly two choices for $\rho_0$,
and that these two are exchanged by conjugation by $\begin{pmatrix} \alpha & 0 \\ 0 & 1 \end{pmatrix}$.
Consequently, either these two choices are preserved by $\rho_0 \mapsto c_M \circ \rho_0 \circ c$,
in which case Example~\ref{example: twistcube CM}  applies directly; or they are exchanged and
Remark~\ref{remark: twisting with denominator} applies.

\begin{table}
\begin{tabular}{c|c} 
Factor & Generators \\
\hline
$[G_8:4]$ & $\begin{pmatrix} -1-i & -3 \\ i & 1+i \end{pmatrix}, \begin{pmatrix} 1 & 3 \\ -1 & -2 \end{pmatrix}$\\ \hline
$G(2,1,2)$ & $\begin{pmatrix}
-1 & -2 \\
0 & 1
\end{pmatrix},
\begin{pmatrix}
-1 & 0 \\
1 & 1
\end{pmatrix}$ \\ \hline
$G(6,6,2)$ & $\begin{pmatrix}
1 & 0 \\
-1 & -1
\end{pmatrix},
\begin{pmatrix}
-1 & -3 \\
1 & 2
\end{pmatrix}$ \\ \hline
$G_4$ & $
\begin{pmatrix}
-1 & -2 \\
1 & 1
\end{pmatrix},
\begin{pmatrix} 
-1& -2 \\
0 & \zeta_3
\end{pmatrix}$ \\ \hline
$G_4'$ & 
$\begin{pmatrix}
-1 & -2 \\
1 & 1
\end{pmatrix},
\begin{pmatrix}
-\zeta_3 & -2\zeta_3 \\
0 & \zeta_3^2
\end{pmatrix}
$
\\ \hline
\end{tabular}
\vspace{6pt}
\caption{Presentations of two-dimensional factors of $\tilde{H}$.
}
\label{table:two-dim generators}
\end{table}

\begin{remark}
Although it is not necessary for our present discussion, for future reference we resolve the ambiguity
in the previous discussion.

\begin{itemize}
\item
For the factors $[G_8:4], G(2,1,2), G(6,6,2)$, we are in a situation similar to point (ii) of
Remark~\ref{remark: twisting and principal polarizations}: the choices for $\rho_0$ are fixed by $c_M$
and exchanged by $c$, so we must apply Remark~\ref{remark: twisting with denominator} rather than
 Example~\ref{example: twistcube CM}.

\item
For the factor $G_4'$, we argue as in point (ii) of
Remark~\ref{remark: twisting and principal polarizations}. 
The group $G_4' \cong \langle 24,3 \rangle$ has a unique 2-Sylow 
subgroup 
\[
P = \left\langle \begin{pmatrix} -1 & -2 \\ 1& 1 \end{pmatrix},
\begin{pmatrix} 1 & -2\zeta_3^3 \\
\zeta_3 & -1 \end{pmatrix} \right\rangle.
\]
The two choices of $\rho_0$, restricted to $P$, have reductions modulo 2 which factor through the quotient $\cyc 2 \times \cyc 2$ of $P$; by labeling the elements of $P$ and of $\F_4$ and keeping track of signs of permutations,
we see that the reductions are interchanged by $c_M$.

For the extension by $\Gal(M/\Q)$ in question, the 2-Sylow subgroup is 
$\langle 16, 13 \rangle$. The quotient of this group by the center of $P$ is abelian, so the reductions are preserved by $c$.

Putting this together, we see that the two choices for $\rho_0$ must be exchanged by 
$\rho_0 \mapsto c_M \circ \rho_0 \circ c$, so we must apply Remark~\ref{remark: twisting with denominator} rather than Example~\ref{example: twistcube CM}.

\item
For the factor $G_4$, recall from the explanation of Table~\ref{table: maximal subgroups U1}
that we get the same projective image by taking the factor $G_4'$ and modifying the action of $c$; we may thus reuse the previous analysis. In this case,
the 2-Sylow subgroup is $\langle 16, 8 \rangle$.
The quotient of this group by the center of $P$ is $\dih 4$, so the reductions are interchanged by $c$.
Consequently, the two choices for $\rho_0$ are preserved by 
$\rho_0 \mapsto c_M \circ \rho_0 \circ c$, so we may apply Example~\ref{example: twistcube CM} directly.

\end{itemize}
\end{remark}

In the cases where $\tilde{H}$ is three-dimensional, we again apply the results of \cite{Fei03}. The only case where $\rho_0$ is not uniquely determined by $\rho$ is $\tilde{H} = G_{25}$, in which case there are two choices for $\rho_0$ which are reducible modulo $1-\zeta_3$. However, the two reductions have invariant subspaces of different dimensions, so the map
$\rho_0 \mapsto c_M \circ \rho_0 \circ c$ fixes them both and Example~\ref{example: twistcube CM} applies directly.

From the previous analysis, it follows that we can apply the construction
of Example~\ref{example: twistcube CM} to obtain an abelian threefold $A$ over $\Q$ with $\ST(A)^0 \simeq \Unitary(1)_3$ and $\ST(A)/\ST(A)^0 \simeq H \rtimes \cyc 2$ \emph{provided} that we can solve the 
Galois embedding problem for the pair $(\tilde{H} \rtimes \cyc 2, \tilde{H})$
and the extension $M/\Q$. We first give a uniform solution for the complex reflection groups.

\begin{lemma} \label{lemma: complex reflection}
Consider any row of Table~\ref{table: maximal subgroups U1}
for which $\tilde{H}$ is a complex reflection group (that is,
$\tilde{H} \neq G(1,1,1) \times [G_{8}:4],
G(3,1,1) \times G_4'$).
Let $V = M^3$ be the representation space on which $\GL_3(M)$ acts.
\begin{enumerate}
\item[(a)]
The invariant ring $R = (\Sym V^\vee)^{\tilde{H}}$ is a polynomial ring in $3$ variables over $M$.
\item[(b)]
The ring $R^c$ is a polynomial ring in $3$ variables over $\Q$.
\end{enumerate}
\end{lemma}
\begin{proof}
To check (a), apply the Chevalley--Shephard--Todd theorem,
which is valid for arbitrary complex reflection groups \cite[\S 7]{ST54}, \cite[Theorem~7.2.1]{Ben11}.
We may deduce (b) from (a) by Galois descent.
\end{proof}

For those rows of Table~\ref{table: maximal subgroups U1} to which Lemma~\ref{lemma: complex reflection} applies,
the affine space corresponding to $R^c$ contains a Zariski-dense open subspace every point of which
defines a solution of the desired embedding problem.

\begin{remark} \label{remark: Hessian as Galois group}
For the various choices of $\tilde{H}$ subsumed in the previous argument, one can also
give more direct solutions of the Galois embedding problem that may be better suited for generating small examples. This is perhaps least obvious when $H$ is the Hessian; in this case,
one has (as pointed out by Noam Elkies) an isomorphism $\tilde{H} \cong \Unitary(3, \F_4)$
via which $\tilde{H}$ occurs as the $2$-division field of a generic Picard curve over $\Q$.
%(see Remark~\ref{remark: Picard mod 2 image}).
\end{remark}

To deal with the remaining cases, we give more direct solutions of the Galois embedding problems in question.
\begin{itemize}
\item
For the case $\tilde{H} = G(1,1,1) \times [G_8:4]$,
we find a Galois extension $L$ of $\Q$ containing $\Q(i)$ with
$\Gal(L/\Q) \cong \langle 48,15 \rangle$, $\Gal(L/\Q(i)) \cong \langle 24, 1 \rangle$. 
By way of motivation, note that the second row of Table~\ref{table: maximal subgroups U1} also represents the abstract group $\langle 48, 15 \rangle$ but with a different presentation: in that case, one takes the semidirect product of $\cyc 3 \times \dih 4$ by $\cyc 2$ acting by an outer automorphism of each factor.
It therefore suffices to take $L$ to be the
compositum of two Galois extensions $L_1/\Q, L_2/\Q$ in which $L_1/\Q$ is an $\sym 3$-extension,
$L_2/\Q$ is a $\dih 8$-extension containing $\Q(i)$ as the fixed subfield of one copy of $\dih 4$,
and the fixed field of the other copy of $\dih 4$ coincides with the quadratic subfield of $L_1/\Q$.

To confirm that this can be achieved, we give an explicit example.
The degree-8 number field
\href{https://www.lmfdb.org/NumberField/8.2.85525504.1}{\texttt{8.2.85525504.1}}
contains $\Q(\sqrt{17})$, its discriminant is $-1$ times a square,
and its Galois group is $\dih 8$; we may thus take $L_2$ to be the Galois closure of this field.
The degree-3 number field
\href{https://www.lmfdb.org/NumberField/3.3.2057.1}{\texttt{3.3.2057.1}}
has discriminant which is $17$ times a square and its Galois group is $\sym 3$;
we may thus take $L_1$ to be the Galois closure of this field.

\item
For the case $\tilde{H} = G(3,1,1) \times G_4'$,
we find a Galois extension $L$ of $\Q$ containing $\Q(\zeta_3)$ with
$\Gal(L/\Q) \cong \langle 144, 125 \rangle$, $\Gal(L/\Q(\zeta_3)) \cong \langle 72,25 \rangle$. 
It suffices to take $L$ to be the compositum of two Galois extensions $L_1/\Q, L_2/\Q$ in which $L_1/\Q$ is an $\sym 3$-extension containing $\Q(\zeta_3)$ and $L_2/\Q$ is an extension containing $\Q(\zeta_3)$ with
$\Gal(L/\Q) \cong \langle 48,33 \rangle$, $\Gal(L/\Q(\zeta_3)) \cong \langle 24, 3 \rangle$.

To confirm that this can be achieved, we give an explicit example.
We take $L_1 = \Q(\zeta_3, 2^{1/3})$. To choose $L_2$, we consult the
Kl\"uners-Malle database \cite{KM01}, which includes the polynomial
\begin{gather*}
x^{16} - x^{15} - 2x^{14} + 8x^{13} + 10x^{12} - 19x^{11} - 4x^{10} + 64x^9 +\\
12x^8 - 94x^7 + 38x^6 + 119x^5 - 64x^4 - 97x^3 + 48x^2 + 11x + 1
\end{gather*}
as an entry for the permutation representation 16T60 (in \Gap\ notation).
Using \Magma, one may confirm that this polynomial is irreducible 
with Galois group $\langle 48, 33 \rangle$, and that the unique quadratic subfield of the splitting field 
is $\Q(\zeta_3)$.
\end{itemize}

\bibliographystyle{amsalpha}

\begin{thebibliography}{ACLLM18}

\bibitem[AAC...18]{AAC18}
P.B. Allen, F. Calegari, A. Caraiani, T. Gee, D. Helm, B.V. Le Hung, J. Newton, P. Scholze, R. Taylor, and
J.A. Thorne, \href{https://arxiv.org/abs/1812.09999v1}{Potential automorphy over CM fields},
arXiv:1812.09999v1 (2018).

\bibitem[ACLLM18]{ACLLM18}
S. Arora, V. Cantoral-Farf\'an, A. Landesman, D. Lombardo, and J.S. Morrow,
\href{https://link.springer.com/article/10.1007/s00209-018-2049-6}{The twisting Sato-Tate group of the curve $y^2=x^8-14x^4+1$},
\textit{Math. Z.} \textbf{290} (2018), 991--1022.

\bibitem[BaK15]{BK15}
G. Banaszak and K.S. Kedlaya,
\href{https://www.jstor.org/stable/pdf/26315458.pdf}{An algebraic Sato-Tate group and Sato-Tate conjecture},
\textit{Indiana Univ. Math. J.} \textbf{64} (2015), 245--274. 

\bibitem[BaK16]{BaK16}
G. Banaszak and K.S. Kedlaya,
\href{http://dx.doi.org/10.1090/conm/663/13348}{Motivic Serre group, algebraic Sato-Tate group and Sato-Tate conjecture}, in \textit{Frobenius Distributions: Lang-Trotter and Sato-Tate Conjectures}, Contemporary Math. 663, Amer. Math. Soc., 2016, 11--44.

\bibitem[BLGG11]{BLGG11}
T. Barnet-Lamb, T. Gee, and D. Geraghty,
\href{http://www.ams.org/journals/jams/2011-24-02/S0894-0347-2010-00689-3/}{The Sato--Tate conjecture for Hilbert modular forms},
\textit{J. Amer. Math. Soc.} \textbf{24} (2011), 411--469.

\bibitem[Ben11]{Ben11}
D. Benson, \href{https://doi.org/10.1017/CBO9780511565809}{\textit{Polynomial Invariants of Finite Groups}}, Cambridge Univ. Press, Cambridge, 2011.

\bibitem[BS02]{BS02}
F. Beukers and C. Smyth,
\href{https://www.maths.ed.ac.uk/~chris/preprints/beukers_smyth.pdf}{Cyclotomic points on curves},
in \textit{Number theory for the millennium, I (Urbana, IL, 2000)},
A.K. Peters, Natick, 2002, 67--85.

\bibitem[CJ76]{CJ76}
J.H. Conway and A.J. Jones,
\href{https://eudml.org/doc/205475}{Trigonometric Diophantine equations (On vanishing sums of roots of unity)},
\textit{Acta Arith.} \textbf{30} (1976), 229--240. 

\bibitem[CFS19]{CFS19} 
Edgar Costa, Francesc Fit\'e, Andrew V. Sutherland, \href{https://www.sciencedirect.com/science/article/pii/S1631073X19302456}{Arithmetic invariants from Sato--Tate moments}, C. R. Math. Acad. Sci. Paris \textbf{357} (2019), 823--826. 

\bibitem[CMSV19]{CMSV19}
E. Costa, N. Mascot, J. Sijsling, and J. Voight,
\href{https://doi.org/10.1090/mcom/3373}{Rigorous computation of the endomorphism ring of a Jacobian},
\textit{Math. Comp.} \textbf{88} (2019), 1303--1339. 

\bibitem[DM82]{DM82}
P. Deligne and J.S. Milne, 
\href{https://link.springer.com/chapter/10.1007/978-3-540-38955-2_4}{\textit{Tannakian Categories}}, 
Lecture Notes in Math. 900, Springer, New York, 1982, 100--228.

\bibitem[Dod84]{Dod84}
B. Dodson,
\href{https://doi.org/10.1090/S0002-9947-1984-0735406-X}{The structure of Galois groups of CM-fields},
\textit{Trans. Amer. Math. Soc.} \textbf{283} (1984),
1--32.

%\bibitem[PSV11]{PSV11}
%D. Plaumann, B. Sturmfels, and C. Vinzant, Quartic curves and their bitangents,
%\textit{J. Symbolic Comput.} \textbf{46} (2011), 712--733. 

\bibitem[Fei03]{Fei03}
W. Feit,
\href{https://doi.org/10.1016/S0021-8693(02)00629-4}{Some integral representations of complex reflection groups},
\textit{J. Algebra} \textbf{260} (2003), 138--153.

\bibitem[FG18]{FG18}
F. Fit\'e and X. Guitart, \href{https://www.ams.org/journals/tran/2018-370-07/S0002-9947-2018-07074-X/home.html}{Fields of definition of elliptic $k$-curves and the realizability of all genus 2 Sato-Tate groups over a number field}, \textit{Trans. Amer. Math. Soc.} \textbf{370} (2018),
4623--4659.

\bibitem[FKRS12]{FKRS12}
F. Fit\'e, K.S. Kedlaya, V. Rotger, and A.V. Sutherland,
\href{https://doi.org/10.1112/S0010437X12000279}{Sato--Tate distributions and Galois endomorphism modules in genus $2$},
\textit{Compos. Math.} \textbf{148} (2012), 1390--1442.

\bibitem[FKS16]{FKS16}
F. Fit\'e, K.S. Kedlaya, and A.V. Sutherland, 
\href{https://dx.doi.org/10.1090/conm/663}{Sato--Tate groups of some weight 3 motives},
in \textit{Frobenius distributions: Lang-Trotter and Sato--Tate conjectures},
Contemp. Math., 663, Amer. Math. Soc., Providence, 2016, 57--101.

\bibitem[FKS]{FKS} F. Fit\'e, K.S. Kedlaya, A.V. Sutherland, Sato--Tate groups of abelian threefolds, preprint. 

\bibitem[FS14]{FS14}
F. Fit\'e and A.V. Sutherland,
\href{http://dx.doi.org/10.2140/ant.2014.8.543}{Sato-Tate distributions of twists of $y^2=x^5-x$ and $y^2=x^6+1$},
\textit{Algebra Number Theory} \textbf{8} (2014), 543--585.

\bibitem[GK17]{GK17}
R. Guralnick and K.S. Kedlaya,
\href{https://link.springer.com/article/10.1007/s40993-017-0088-4}{Endomorphism fields of abelian varieties},
\textit{Res. Number Theory} \textbf{3}:22 (2017).

\bibitem[HLP00]{HLP00}
E.W. Howe, F. Lepr\'evost, and B. Poonen, 
\href{http://www-math.mit.edu/~poonen/papers/large.pdf}{Large torsion subgroups of split Jacobians of curves of genus two or three},
\textit{Forum Math.} \textbf{12} (2000), 315--364.

\bibitem[Kan11]{Kan11}
E. Kani, \href{https://doi.org/10.1007/s13348-010-0029-1}{Products of CM elliptic curves}, \textit{Collect. Math.} \textbf{62} (2011),
297--339.

\bibitem[KS08]{KS08}
K.S. Kedlaya and A.V. Sutherland,
\href{https://link.springer.com/chapter/10.1007/978-3-540-79456-1_21}{Computing L-series of hyperelliptic curves},
in \textit{Algorithmic Number theory: 8th International Symposium, ANTS-VIII (Banff, Canada, May
2008)},
Lect. Notes Comp. Sci. 5011, Springer, Berlin, 2008,
312--326.

\bibitem[KS09]{KS09}
K.S. Kedlaya and A.V. Sutherland, 
\href{https://doi.org/10.1090/conm/487/09529}{Hyperelliptic curves, $L$-polynomials, and random matrices},
in \textit{Arithmetic, Geometry, Cryptography, and Coding Theory: International Conference, November
5--9, 2007, CIRM, Marseilles, France}, 
Contemp. Math. 487, Amer. Math. Soc., Providence, 2009, 119--162.

\bibitem[K\i l16]{Kil16}
P. K\i l\i \c{c}er,
\href{https://openaccess.leidenuniv.nl/handle/1887/41145}{The CM class number one problem for curves},
thesis, Universit\'e de Bordeaux, 2016, \url{https://openaccess.leidenuniv.nl/handle/1887/41145}.

\bibitem[KM01]{KM01}
J. Kl\"uners and G. Malle,
\href{https://doi.org/10.1112/S1461157000000851}{A database for field extensions of the rationals},
\textit{LMS J. Comput. Math.} \textbf{4} (2001), 182--196;
\url{http://galoisdb.math.upb.de}.

\bibitem[LMFDB]{LMFDB}
The LMFDB Collaboration, \href{https://www.lmfdb.org}{\textit{The L-functions and Modular Forms Database}}, \url{https://www.lmfdb.org}, 2018
(accessed Oct 2019).

\bibitem[Ma11]{Ma11}
S. Ma, \href{http://www.numdam.org/item/AIF_2011__61_2_717_0/}{Decompositions of an Abelian surface and quadratic forms},
\textit{Ann. Inst. Fourier} \textbf{61} (2011), 717--743.

\bibitem[MM10]{MM10}
I. Marin and J. Michel,
\href{https://www.ams.org/journals/ert/2010-14-21/S1088-4165-2010-00380-5/S1088-4165-2010-00380-5.pdf}{Automorphisms of complex reflection groups},
\textit{Rep. Theory} \textbf{14} (2010), 747--788.

\bibitem[MBD61]{MBD61}
G.A. Miller, H.F. Blichfeldt, and L.E. Dickson,
\href{https://link.springer.com/article/10.1007/BF01696919}{\textit{Theory and Applications of Finite Groups}},
Dover, New York, 1961.

\bibitem[Mum69]{Mum69}
D. Mumford,
\href{http://link.springer.com/article/10.1007/BF01350672}{A note on Shimura's paper ``Discontinuous subgroups and abelian varieties''}, \textit{Math. Annalen} \textbf{181} (1969) 345--351.

\bibitem[Oo87]{Oo87} F. Oort,
\href{https://doi.org/10.1016/B978-0-12-348032-3.50007-4}{Endomorphism algebras of abelian varieties}, in \textit{Algebraic geometry and commutative algebra}, 469-502 (1988).

\bibitem[ST54]{ST54}
G.C. Shephard and J.A. Todd,
\href{https://doi.org/10.4153/CJM-1954-028-3}{Finite unitary reflection groups},
\textit{Canad. J. Math.} \textbf{6} (1954), 274--304.

\bibitem[Shi71]{Shi71}
G. Shimura,
\href{https://www.jstor.org/stable/1970768}{On the zeta-function of an abelian variety with complex multiplication}, \textit{Ann. of Math.} \textbf{94} (1971), 504--533.

\bibitem[SM74]{SM74}
T. Shioda and N. Mitani, \href{https://link.springer.com/chapter/10.1007/BFb0066163}{Singular abelian surfaces and binary quadratic forms},
in \textit{Classification of algebraic varieties and compact complex manifolds},
Lecture Notes in Math. 412, Springer, 1974, 259--287.

\bibitem[Sut16]{Sut16}
A.V. Sutherland,
\href{https://doi.org/10.1090/conm/740/14904}{Sato--Tate distributions}, in \textit{Analytic Methods in Arithmetic Geometry: Arizona Winter School, March 12--16, 2016, The University of Arizona, Tucson, AZ}, Contemp. Math. 740, Amer. Math.Soc., Providence 2019, 197--248..

\bibitem[Sut18a]{Sut18a}
A.V. Sutherland,
\href{https://msp.org/obs/2019/2-1/p26.xhtml}{Fast Jacobian arithmetic for hyperelliptic curves of genus $3$}, in \textit{Thirteenth Algorithmic Number Theory Symposium (ANTS XIII)}, Open Book Series \textbf{2} (2019), 425--442.

\bibitem[Sut18b]{Sut18b}
A.V. Sutherland,
\href{https://msp.org/obs/2019/2-1/p27.xhtml}{A database of nonhyperelliptic genus 3 curves over $\Q$},
in \textit{Thirteenth Algorithmic Number Theory Symposium (ANTS XIII)}, Open Book Series \textbf{2} (2019), 443--459.

\bibitem[Upt09]{Upt09}
M. Upton, \href{https://core.ac.uk/download/pdf/82038456.pdf}{Galois representations attached to Picard curves},
\textit{J. Algebra} \textbf{322} (2009), 1038--1059.
 
\bibitem[Vas04]{Vas04}
A. Vasiu,
\href{https://link.springer.com/article/10.1007/s00229-004-0465-x}{Surjectivity criteria for $p$-adic representations, part II},
\textit{Manuscripta Math.} \textbf{114} (2004), 399--422. 

\bibitem[Za00]{Za00}
Y.G. Zarhin,
\href{https://www.sciencedirect.com/science/article/pii/S0022314X05001216}{Hyperelliptic curves without complex multiplication},
\textit{Math. Res. Lett.} \textbf{7} (2000), 123--132.

\end{thebibliography}

\end{document}